\documentclass{article}
\usepackage[utf8]{inputenc}

\usepackage{amsmath,amsthm}
\usepackage{amsfonts}
\usepackage{amssymb}
\usepackage{enumitem}
\usepackage{makeidx}
\usepackage{tikz}
\usepackage{tikz-cd}
\usepackage{wrapfig}
\usepackage{xcolor}
\usepackage{mathtools}
\usepackage{graphicx}
\usepackage{caption}
\usepackage{comment}
\usepackage{scalerel}
\usepackage{hyperref}
\usepackage{cleveref}
\usepackage[normalem]{ulem}
\usepackage{authblk}

\newcommand*{\red}{\textcolor{red}}

\newcommand*{\blue}{\textcolor{blue}}

\newcommand{\s}[1]{\mathbf{#1}}

\newtheorem{theorem}{Theorem}[section]
\newtheorem{theorem*}[theorem]{*Theorem}
\newtheorem{prop}[theorem]{Proposition}
\newtheorem{lema}[theorem]{Lemma}

\newtheorem{claim}[theorem]{Claim}

\newtheorem{question}[theorem]{Question}

\newtheorem{cor}[theorem]{Corollary}

\newtheorem{prop*}[theorem]{*Proposition}
\newtheorem{lema*}[theorem]{*Lemma}
\newtheorem{cor*}[theorem]{*Corollary}

\theoremstyle{definition}
\newtheorem{definition}[theorem]{Definition}
\newtheorem{definition*}[theorem]{*Definition}

\newtheorem{defs*}[theorem]{*Definitions}

\newtheorem{remark}[theorem]{Remark}

\newtheorem{innercustomgeneric}{\customgenericname}
\providecommand{\customgenericname}{}
\newcommand{\newcustomtheorem}[2]{%
  \newenvironment{#1}[1]
  {%
   \renewcommand\customgenericname{#2}%
   \renewcommand\theinnercustomgeneric{##1}%
   \innercustomgeneric
  }
  {\endinnercustomgeneric}
}

\newcustomtheorem{customthm}{Theorem}
\newcustomtheorem{customprop}{Proposition}

\title{An inverse of Furstenberg's correspondence principle and applications to van der Corput sets} 
\author[1]{
Saúl Rodríguez Martín}
\affil[1]{The Ohio State University\\rodriguezmartin.1@osu.edu}
\date{\vspace{-20pt}}

\begin{document}
\maketitle

\begin{abstract}
We obtain an inverse of Furstenberg's correspondence principle in the 
setting of countable cancellative, amenable semigroups. Besides being of intrinsic interest on its own, this result allows us to answer a variety of questions concerning sets of recurrence and van der Corput (vdC) sets, which were posed by Bergelson and Lesigne \cite{BL}, Bergelson and Ferré Moragues \cite{BF}, Kelly and Lê \cite{KL}, and Moreira \cite{Mor}. We also prove a spectral characterization of vdC sets and prove some of their basic properties in the context of countable amenable groups.

Several results in this article were independently found by Sohail Farhangi and Robin Tucker-Drob, see \cite{FT}.
\end{abstract}


\section{Introduction}

In this article, we establish an inverse of Furstenberg’s correspondence principle in the framework of countable, discrete amenable semigroups. Beyond its intrinsic significance, this result enables us to answer a range of open questions posed by Bergelson and Lesigne \cite{BL}, Bergelson and Ferré Moragues \cite{BF}, Kelly and Lê \cite{KL}, and Moreira \cite{Mor}.

We first give some context for Furstenberg's correspondence principle.

\begin{definition}
\label{DefAmenableGroup}
Let $G$ be a countable group. A \textit{(left) F{\o}lner sequence} $F=(F_N)_{N\in\mathbb{N}}$ in $G$ is a sequence of finite sets $F_N\subseteq G$ such that, for all $g\in G$, $\lim_{N\to\infty}\frac{|gF_N\Delta F_N|}{|F_N|}=0$. We say $G$ is \textit{(left) amenable} if it has a F{\o}lner sequence.\footnote{Throughout this article we will use only left amenability and left F{\o}lner sequences, so we will omit the adjective `left'.}
\end{definition}

\begin{definition}
\label{DefUBD}
Let $E$ be a subset of a countable amenable group $G$. The \textit{upper density} of $E$ along a F{\o}lner sequence $F=(F_N)$ is defined by \begin{equation*}
\overline{d}_F(E):=\limsup_{N\to\infty}\frac{|E\cap F_N|}{|F_N|}.
\end{equation*}
If the $\limsup$ is actually a limit, we call it $d_F(E)$, the \textit{density} of $E$ along $(F_N)$. The upper Banach density of $E$, $d^*(E)$, is defined by
\begin{equation*}
d^*(E)=\sup\left\{\overline{d}_F(E);F\text{ F{\o}lner sequence in }G\right\}.
\end{equation*}
\end{definition}
If $A\subseteq\mathbb{N}=\{1,2,\dots\}$, we denote by $\overline{d}(A)$ the upper density of $A$ with respect to the F{\o}lner sequence $F_N=\{1,2,\dots,N\}$ in $\mathbb{Z}$.

Szemerédi's theorem on arithmetic progressions states that any set $A$ of natural numbers with $\overline{d}(A)>0$ contains arithmetic progressions of length $k$ for all $k\in\mathbb{N}$. Szemerédi proved this theorem in \cite{Sze} using combinatorial methods. In \cite{Fu77}, Furstenberg gave a new, ergodic proof of Szemerédi's theorem, see \Cref{ErgSze} below. Throughout this article, we say that $(X,\mathcal{B},\mu,(T_s)_{s\in S})$ is a \textit{measure preserving system}, m.p.s. for short, if $(X,\mathcal{B},\mu)$ is a probability space and $(T_s)_{s\in S}$ is an action of a semigroup $S$ on $X$ by measure preserving maps $T_s:X\to X$ (so $T_s\circ T_t=T_{st}$). Similarly, we say $(X,\mathcal{B},\mu,T)$ is a m.p.s. when $T:X\to X$ is a measure preserving map of the probability space $(X,\mathcal{B},\mu)$.

\begin{theorem}
\label{ErgSze}
Let $(X,\mathcal{B},\mu,T)$ be a m.p.s. and let $C\in\mathcal{B}$ satisfy $\mu(C)>0$. Then for all $k\in\mathbb{N}$ there is some $n\in\mathbb{N}$ such that 
\begin{equation*}
\mu(T^{-n}C\cap T^{-2n}C\cap\cdots\cap T^{-kn}C)>0.
\end{equation*}
\end{theorem}
The method that Furstenberg used to derive Szemerédi's theorem from \Cref{ErgSze} is nowadays called Furstenberg's correspondence principle. We state it in the setting of amenable groups:

\begin{theorem}[Furstenberg's correspondence principle, cf. {\cite[Theorem 1.8]{Be5}}]\label{FCP}
Let $G$ be a countable amenable group with a F{\o}lner sequence $F=(F_N)$. For any $A\subseteq G$ there is a m.p.s. $(X,\mathcal{B},\mu,(T_g)_{g\in G})$ and $B\in\mathcal{B}$ with $\mu(B)=\overline{d}_F(A)$ such that, for all $k\in\mathbb{N}$ and $h_1,\dots,h_k\in G$,
\begin{equation*}
\overline{d}_F(h_1A\cap\dots\cap h_kA)\geq\mu(T_{h_1}(B)\cap\dots\cap T_{h_k}(B)),
\end{equation*}
and for all $h_1,\dots,h_k$ such that $d_F(h_1A\cap\dots\cap h_kA)$ exists,
\begin{equation}
\label{FCPEq}
d_F(h_1A\cap\dots\cap h_kA)=\mu(T_{h_1}(B)\cap\dots\cap T_{h_k}(B)).
\end{equation}
\end{theorem}

In particular, when $G=\mathbb{Z}$ and $F_N=\{1,\dots,N\}$, \Cref{FCP} and \Cref{ErgSze} imply the following result,
which in turn implies Szemerédi's theorem.

\begin{theorem}
\label{NumSze}
Let $A\subseteq\mathbb{N}$ satisfy $\overline{d}(A)>0$. Then for all $k\in\mathbb{N}$ there is some $n\in\mathbb{N}$ such that 
\begin{equation*}
\overline{d}((A-n)\cap(A-2n)\cap\cdots\cap(A-kn))>0.
\end{equation*}
\end{theorem}

\Cref{FCPEq} naturally leads to the question of whether given a m.p.s. $(X,\mathcal{B},\mu,(T_g)_{g\in G})$, 
a set $B\in\mathcal{B}$ and a F{\o}lner sequence $F$ in $G$, there is some $A\subseteq G$ satisfying (\ref{FCPEq}) for all $k,h_1,\dots,h_k$.
The answer is yes:
\begin{theorem}[Inverse Furstenberg correspondence principle]
\label{IFC}
Let $G$ be a countably infinite amenable group with a F{\o}lner sequence $(F_N)$. For every m.p.s. $(X,\mathcal{B},\mu,(T_g)_{g\in G})$ and every $B\in\mathcal{B}$ there exists a subset $A\subseteq G$ such that for all $k\in\mathbb{N}$ and $h_1,\dots,h_k\in G$ we have
\begin{equation}\label{EqStrongFurstints}
d_F\left(h_1A\cap\dots\cap h_kA\right)
=
\mu\left(T_{h_1}B\cap\dots\cap T_{h_k}B\right).
\end{equation}
\end{theorem}

In \Cref{SecRrg} we prove that \Cref{IFC} actually holds for cancellative amenable semigroups (see \Cref{IFCSemigroups}). 

\begin{remark}
A special case of \Cref{IFC}, which deals with the case $G=\mathbb{Z}$ and $F_N=\{1,\dots,N\}$, was obtained by Fish and Skinner in \cite[Theorem 1.4]{FS}. \Cref{IFC} answers a question of Moreira \cite[Section 6]{Mor}, which was formulated for countable abelian groups.
Another special case of \Cref{IFC} is \cite[Theorem 5.1]{BF2}, where the authors assume that the action $(T_g)_{g\in G}$ is ergodic and obtain a variant of \Cref{EqStrongFurstints} by passing to a subsequence of $(F_N)$. 
Farhangi and Tucker-Drob have independently obtained (by a different method) a version of \Cref{IFC}, and its generalization, \Cref{ThmFinAvs2.0} (see \cite[Theorem 1.2]{FT}).
\end{remark}

\begin{remark}
\label{IFCUnions}
Due to the algebraic nature of Furstenberg's correspondence principle, \Cref{IFC} admits a more general version where some of the involved sets are replaced by their complements, or some intersections are replaced by unions (a version of Furstenberg's correspondence principle dealing with unions and complements was established in \cite[Theorem 2.3]{BF}, see also \cite[Theorem 2.3]{BBF}). See \Cref{76iuyjhre5rytghf,IFCU} for more details.
\end{remark}

Our \Cref{IFC} was motivated by some open questions in the theory of van der Corput (vdC) sets, which \Cref{IFC} allows us to resolve. The notion of vdC set was introduced by Kamae and Mendès France in \cite{KM}, in connection with the theory of uniform distribution of sequences in $\mathbb{T}=\mathbb{R}/\mathbb{Z}$. Recall that a sequence $(x_n)_{n\in\mathbb{N}}$ in $\mathbb{T}$ is \emph{uniformly distributed mod $1$} (u.d. mod $1$) if for any continuous function $f:\mathbb{T}\to\mathbb{C}$ we have
\begin{equation*}
\lim_{N\to\infty}\frac{1}{N}\sum_{n=1}^Nf(x_n)=\int_{\mathbb{T}}
fdm,
\end{equation*}
where $m$ is Lebesgue measure.
\begin{definition}[{\cite[Page 1]{KM}}]\label{defequivdCZ}
A set $H\subseteq\mathbb{N}:=\{1,2,\dots\}$ is a \textit{van der Corput set} (vdC set) if, for any sequence $(x_n)_{n\in\mathbb{N}}$ in $\mathbb{T}=\mathbb{R}/\mathbb{Z}$ such that $(x_{n+h}-x_n)_{n\in\mathbb{N}}$ is u.d. mod $1$ for all $h\in H$, the sequence $(x_n)_{n\in\mathbb{N}}$ is itself u.d. mod $1$.
\end{definition}

Both the notions of u.d. mod $1$ and of vdC set naturally extend to more general averaging schemes and indeed to amenable semigroups, as suggested in \cite[Section 4.2]{BL}.

\begin{definition}
\label{DefFudmod1}
Let $F=(F_N)$ be a F{\o}lner sequence in a countable amenable group $G$. We say that a sequence $(x_g)_{g\in G}$ in $\mathbb{T}$ is \textit{$F$-u.d. mod $1$} if for any continuous function $f:\mathbb{T}\to\mathbb{\mathbb{C}}$ we have
\begin{equation*}
\lim_{N\to\infty}\frac{1}{|F_N|}\sum_{g\in F_N}f(x_g)=\int_{\mathbb{T}}fdm.
\end{equation*}
\end{definition}

For general (not necessarily abelian) groups $G$, we will denote their identity element by $e_G$, or just $e$ if the group is clear from the context.

\begin{definition}[cf. {\cite[Page 44]{BL}}]\label{DefFvdC}
Let $F=(F_N)$ be a F{\o}lner sequence in a countable amenable group $G$. We say a subset $H$ of $G\setminus\{e\}$ is $F$-vdC if any sequence $(x_g)_{g\in G}$ in $\mathbb{T}$ such that $(x_{hg}-x_g)_{g\in G}$ is $F$-u.d. mod $1$ for all $h\in H$, is itself $F$-u.d. mod $1$.
\end{definition}
In \cite[Section 4.2]{BL} the authors posed the question whether, for any F{\o}lner sequence $F$ in $\mathbb{Z}$, a subset $H\subseteq\mathbb{N}$ is $F$-vdC if and only if it is a vdC set. By using an amplified version of \Cref{IFC} (\Cref{ThmFinAvsGen}) we show that the answer is yes by giving the following characterization of $F$-vdC sets which does not depend on the F{\o}lner sequence:

\begin{theorem}\label{5trfgdtrgfd}
Let $G$ be a countably infinite amenable group with a F{\o}lner sequence $F=(F_N)$. A set $H\subseteq G\setminus\{e\}$ is $F$-vdC in $G$ if and only if for any m.p.s. $(X,\mathcal{B},\mu,(T_g)_{g\in G})$ and for any function $f\in L^\infty(\mu)$,
\begin{equation*}
\int_{X}f(T_hx)\cdot\overline{f(x)}d\mu(x)=0\text{ for all }h\in H\textup{ implies }\int_{X}fd\mu=0.
\end{equation*} 
\end{theorem}

It is worth mentioning that the condition given in \Cref{5trfgdtrgfd} makes sense for any countable (not necessarily amenable) group\footnote{The definition makes sense for any (discrete) group, even if it is not countable. We will not be interested in uncountable discrete groups in this article, but many of the properties of vdC sets generalize to this setting.}. This leads to the following general definition:
\begin{definition}\label{DefGvdC}
Let $G$ be a countable group. We will say that $H\subseteq G\setminus\{e\}$ is \textit{vdC in $G$} if,
for every m.p.s. $(X,\mathcal{B},\mu,(T_g)_{g\in G})$ and every $f\in L^\infty(\mu)$,
\begin{equation*}
\int_{X}f(T_hx)\cdot\overline{f(x)}d\mu(x)=0\text{ for all }h\in H\text{ implies }\int_{X}fd\mu=0.
\end{equation*} 
\end{definition}

From \Cref{DefGvdC} it follows that any vdC set in a countable group $G$ is of (measurable) recurrence, in the sense that for every m.p.s. $(X,\mathcal{B},\mu,(T_g)_{g\in G})$ and for every $B\in\mathcal{B}$ such that $\mu(B)>0$, we have $\mu(B\cap T_hB)>0$ for some $h\in H$. Indeed, if instead of allowing any $f\in L^\infty(\mu)$ we restrict our attention to characteristic functions (or positive functions), \Cref{DefGvdC} becomes the definition of set of recurrence. This supports the idea implied by \cite[Section 3.2]{BL} that sets of recurrence are a `positive version' of vdC sets.

Taking definition \Cref{DefGvdC} as the starting point, we will prove in \Cref{SecProps}, with the help of \Cref{5trfgdtrgfd}, several properties of vdC sets. For example, we show that the family of vdC sets in a countable group $G$ has the partition regularity property, that is, if $H\subseteq G$ is a vdC set and $H=H_1\cup H_2$, then either $H_1$ or $H_2$ is a vdC set. 
We also prove that any vdC set in an amenable group contains two disjoint vdC sets. And finally, we study the behaviour of vdC sets in subgroups and under group homomorphisms and give some (non-)examples of vdC sets. 
Some of the results of \Cref{SecProps} are generalizations of results which were obtained for $\mathbb{Z}$ and $\mathbb{Z}^d$ in \cite{Ru} and \cite{BL}.

In \Cref{SpeCrit}, which will be presently formulated, we give a spectral characterization of vdC subsets of any countable abelian group $G$ 
(another, very elegant, proof of the result can be found in \cite{FT}). 
This provides a generalization of similar results obtained in \cite{KM,BL,Ru} for vdC subsets of $\mathbb{Z}$ or $\mathbb{Z}^d$ for $d\geq2$ (see e.g. \cite[Theorem 1.8.]{BL}). 
Spectral characterizations (in the setup of abelian groups) are useful both for proving properties of vdC sets and for finding (non-)examples of them. For example, Bourgain used in \cite{Bo} a version of \Cref{SpeCrit} for $G=\mathbb{Z}$ to construct a set of recurrence which is not a vdC set.

Given a discrete, countable abelian group $G$ we denote by $0$ (instead of $e$) the identity element of $G$ and by $\widehat{G}$ its Pontryagin dual (which is compact and metrizable, see \cite[Theorems 1.2.5, 2.2.6]{RuFA}), with identity $1_{\widehat{G}}$. For any Borel probability measure $\mu$ in $\widehat{G}$ we denote the Fourier coefficients of $\mu$ by $\widehat{\mu}(h)=\int_{\widehat{G}}\gamma(h)d\mu(\gamma),h\in G$.

\begin{theorem}[{cf. \cite[Theorem 1.8]{BL}}]\label{SpeCrit}
Let $G$ be a countable abelian group. A set $H\subseteq G\setminus\{0\}$ is a vdC set in $G$ iff any Borel probability measure $\mu$ in $\widehat{G}$ with $\widehat{\mu}(h)=0\;\forall h\in H$ satisfies $\mu(\{0\})=0$.
\end{theorem}

The structure of the article is as follows. 
In \Cref{SecIntro2} we use \Cref{IFC} (and a more general version of it) to answer some questions from \cite{BL,BF,Mor}. 
In \Cref{SecRrg} we first prove \Cref{ThmFinAvsGen}, an amplified version of \Cref{IFC}. 
We also obtain \Cref{IFCSemigroups}, a general version of \Cref{IFC} for cancellative amenable semigroups. 
In \Cref{SecCharact} we prove \Cref{ThmvdC}, an amplified version of \Cref{5trfgdtrgfd} which contains several characterizations of vdC sets. 
In \Cref{SecSpect} we establish a spectral characterization of vdC sets in countable abelian groups, \Cref{SpeCrit}. In \Cref{SecProps} we prove fundamental properties of vdC sets in amenable groups, such as for example partition regularity. 
The main result in \Cref{SecConv} is \Cref{564rtyfgdv}, which concerns the relationship between the set of possible Cesaro averages of sequences taking values in a compact set $D\subseteq\mathbb{C}$, and in the convex hull of $D$. We deduce several results from \Cref{564rtyfgdv}, including an affirmative answer to a question of Kelly and Lê.

\paragraph{Acknowledgements. } 

Thanks to Vitaly Bergelson for his guidance while writing this article, and for some interesting discussions and suggestions.

Several months before uploading this article, it came to our attention that Sohail Farhangi and Robin Tucker-Drob had been independently studying the topic of vdC sets. Their paper \cite{FT} contains a long list of characterizations of vdC sets, including the ones from \Cref{5trfgdtrgfd} and \Cref{SpeCrit}. 

We appreciate Farhangi’s input, suggestions, and quick review of the paper. He brought \cite[Theorem 5.2]{DHZ} to our attention, which allowed us to simplify the proof of \Cref{SpeCrit} and to state \Cref{ThmFinAvsGen} for all amenable groups instead of only monotileable ones. He also noticed how \Cref{564rtyfgdv} can be used to answer a question of Kelly and Lê.

We gratefully acknowledge support from the grants BSF 2020124 and NSF CCF AF 2310412.
\section{Some applications of the inverse correspondence principle}
\label{SecIntro2}

In this section we use \Cref{IFC} (and a more general result, \Cref{ThmFinAvs2.0}), to answer several questions from the literature. 
We first answer\footnote{S. Farhangi proved in his dissertation \cite{Fa} that every nice vdC set is a set of nice recurrence for the F{\o}lner sequence  $F_N = \{1, \dots, N\}$, thereby addressing Bergelson and Lesigne’s original question. We give a different proof and generalize this result to amenable groups.} a question of Bergelson and Lesigne in the general context of countable amenable groups, by proving that every nice vdC set is a set of nice recurrence. The second of these two notions was introduced in \cite{Be2} for subsets of $\mathbb{Z}$, although we use the slightly different definition given in \cite{BL} (we check that both are equivalent in \Cref{refwds54rtefdvhbdhtt}).

\begin{definition}[cf. {\cite[Definition 2.2]{Be2}}]\label{DefNiceRec}
Let $G$ be a group. A subset $H$ of $G$ is a \textit{set of nice recurrence} if for any m.p.s. $(X,\mathcal{B},\mu,(T_g)_{g\in G})$ and for any $B\in\mathcal{B}$,
\begin{equation*}
\label{3t4ouerifsfdjf}
\mu(B)^2\leq\limsup_{h\in H}\mu(B\cap T_hB).\footnote{Here we adopt the notation $\limsup_{h\in H}x_h
:=
\inf_{H_0\subseteq H\text{ finite}}\sup_{h\in H\setminus H_0}x_h$.}
\end{equation*}
\end{definition}
In order to motivate the definition of nice vdC sets, we first state a theorem of Ruzsa which characterizes vdC sets in terms of Cesaro averages:

\begin{theorem}[{cf. \cite[Theorem 1]{Ru}}]\label{defvcDZ}
A set $H\subseteq\mathbb{N}$ is a vdC set iff for any sequence $(z_n)_{n\in\mathbb{N}}$ of complex numbers with $|z_n|\leq1$ for all $n$,
\begin{equation*}
\quad\lim_{N\to\infty}\frac{1}{N}\sum_{n=1}^Nz_{h+n}\overline{z_n}=0\textup{ for all }h\in H\textup{ implies }
\lim_{N\to\infty}\frac{1}{N}\sum_{n=1}^Nz_n=0.
\end{equation*}
\end{theorem}

\begin{definition}[cf. {\cite[Definition 10]{BL}}]\label{DefNicevdC}
Let $G$ be a countable amenable group with a F{\o}lner sequence $(F_N)$. A subset $H$ of $G\setminus\{e\}$ is \emph{nice $F$-vdC} if for any sequence $(z_g)_{g\in G}$ in $\mathbb{D}$, 
\begin{equation*}
\limsup_N\left|\frac{1}{|F_N|}\sum_{g\in F_N}z_g\right|^2
\leq
\limsup_{h\in H}
\limsup_N\left|\frac{1}{|F_N|}
\sum_{g\in F_N}z_{hg}\overline{z_g}\right|
\end{equation*}
\end{definition}

\Cref{defvcDZ} implies that in $\mathbb{Z}$ (or $\mathbb{N}$), nice vdC sets are vdC sets. This result is true for any amenable group, as implied by \Cref{5trfgdtrgfd} and \Cref{CharactNiceVdC} below. Note that \Cref{DefNicevdC} and \Cref{DefNiceRec} are expressed in different settings: \Cref{DefNicevdC} is about Cesaro averages of sequences of complex numbers, while \Cref{DefNiceRec} is about integrals. We now translate each of these definitions to the setting of the other one:

\begin{prop}\label{trefdvrefds}
Let $G$ be a countable amenable group with a F{\o}lner sequence $F=(F_N)$. Then a subset $H\subseteq G$ is a set of nice recurrence iff for any $E\subseteq G$ we have
\begin{equation*}
\overline{d_F}(E)^2\leq\limsup_{h\in H}\overline{d_F}(E\cap hE).
\end{equation*}
\end{prop}

\begin{prop}\label{CharactNiceVdC}
Let $G$ be a countable amenable group with a F{\o}lner sequence $(F_N)$. Then a subset $H$ of $G\setminus\{e\}$ is nice $F$-vdC iff for any m.p.s. $(X,\mathcal{B},\mu,(T_g)_{g\in G})$ and any $f\in L^\infty(X,\mu)$ we have
\begin{equation*}
\left|\int_Xfd\mu\right|^2
\leq
\limsup_{h\in H}\left|\int_Xf(T_hx)\overline{f(x)}d\mu(x)\right|.
\end{equation*}
\end{prop}
\begin{remark}
\Cref{CharactNiceVdC} implies that the notion of nice $F$-vdC set is independent of the F{\o}lner sequence.
\end{remark}

\begin{proof}[Proof of \Cref{trefdvrefds}]\leavevmode
\begin{enumerate}
    \item[$\implies$] Suppose there is some set $E\subseteq G$ such that $\overline{d_F}(E)^2>\limsup_{h\in H}\overline{d_F}(E\cap hE)$. Then there is a F{\o}lner subsequence $F'$ of $F$ such that $d_{F'}(E)$ exists, $d_{F'}(E\cap hE)$ exists for all $h$ and $d_{F'}(E)^2>\limsup_{h\in H}d_{F'}(E\cap hE)$. But by \Cref{FCP} there is some m.p.s. $(X,\mathcal{B},\mu,(T_g)_{g\in G})$ and some $B\in\mathcal{B}$ such that $\mu(B)=d_{F'}(E)$ and $\mu(B\cap T_hB)=d_{F'}(E\cap hE)$ for all $h\in G$, thus $\mu(B)^2>\limsup_{h\in H}\mu(B\cap T_hB)$.

    \item[$\impliedby$] Suppose there is a m.p.s. $(X,\mathcal{B},\mu,(T_g)_{g\in G})$ and some $B\in\mathcal{B}$ such that $\mu(B)^2>\limsup_{h\in H}\mu(B\cap T_hB)$. By \Cref{IFC} there is some set $E\subseteq G$ such that $\mu(B)=d_{F}(E)$ and $\mu(B\cap T_hB)=d_{F}(E\cap hE)$ for all $h\in G$, thus $d_{F'}(E)^2>\limsup_{h\in H}d_{F'}(E\cap hE)$.\qedhere
\end{enumerate}
\end{proof}

The proof of \Cref{CharactNiceVdC} is completely analogous to that of \Cref{trefdvrefds}, except that we will need to use a result slightly more general than \Cref{IFC}. 
Note that the sets $A,B$ from \Cref{IFC} can be identified with their characteristic functions, that is, $\{0,1\}$-valued functions. \Cref{IFC} corresponds to the specific case of \Cref{ThmFinAvs2.0} corresponding to $D=\{0,1\}$ and functions $p$ of the form $p(z_1,\dots,z_n)=\prod_{i=1}^nz_i$.
\begin{theorem}\label{ThmFinAvs2.0}
Let $G$ be a countably infinite amenable group with a F{\o}lner sequence $(F_N)$ and let $D\subseteq\mathbb{C}$ be compact. Then for any m.p.s. $(X,\mathcal{B},\mu,(T_g)_{g\in G})$ and any measurable function $f:X\to D$ there is a sequence $(z_g)_{g\in G}$ of complex numbers in $D$ such that, for all $j\in\mathbb{N}$, $h_1,\dots,h_j\in G$ and all continuous functions $p:D^j\to\mathbb{C}$,
\begin{equation}\label{r4fdfrdedfdfsdsf}
\lim_N\frac{1}{|F_N|}\sum_{g\in F_N}p(z_{h_1g},\dots,z_{h_jg})=\int_Xp(f(T_{h_1}x),\dots,f(T_{h_j}x))d\mu.
\end{equation}
Conversely, given a sequence $(z_g)_{g\in G}$ in $D$ there is a m.p.s. $(X,\mathcal{B},\mu,(T_g)_{g\in G})$ and a measurable function $f:X\to D$ such that, for all $j\in\mathbb{N},h_1,\dots,h_j\in G$ and $p:D^j\to\mathbb{C}$ continuous, \Cref{r4fdfrdedfdfsdsf} holds if the left hand side limit exists.
\end{theorem}
\Cref{ThmFinAvs2.0} pertains to sequences with values in a compact subset of $\mathbb{C}$, but in some cases one may adapt it to unbounded sequences, see \cite[Theorem 3.3]{FT}.

\begin{proof}[Proof of \Cref{CharactNiceVdC}]\leavevmode
\begin{enumerate}
    \item[$\implies$]
Suppose for contradiction that there is a m.p.s. $(X,\mathcal{B},\mu,(T_g)_{g\in G})$ and some measurable function $f:X\to\mathbb{D}$ such that 
\begin{equation*}
\left|\int_Xfd\mu\right|^2
>
\limsup_{h\in H}\left|\int_Xf(T_hx)\overline{f(x)}d\mu(x)\right|.
\end{equation*}
By \Cref{ThmFinAvs2.0} there is some sequence $(z_g)_{g\in G}$ such that
\begin{align*}
\lim_N\frac{1}{|F_N|}\sum_{g\in F_N}z_g&=\int_Xf(x)d\mu\\
\lim_N\frac{1}{|F_N|}\sum_{g\in F_N}z_{hg}\overline{z_g}&=\int_Xf(T_{h}x)\overline{f(x)}d\mu\textup{ for all }h\in H,
\end{align*}
so $H$ is not a nice $F$-vdC set.

    \item[$\impliedby$] If $H$ is not nice $F$-vdC then for an adequate sequence $(z_g)_{g\in G}$ in $\mathbb{D}$ and some subsequence $F'=(F_N')_{N\in\mathbb{N}}$ of $F$ we have 
\begin{equation*}
\lim_N\left|\frac{1}{|F_N'|}\sum_{g\in F_N'}z_g\right|^2
>
\limsup_{h\in H}\lim_N\left|\frac{1}{|F_N'|}\sum_{g\in F_N'}z_{hg}\overline{z_g}\right|.
\end{equation*}
Applying again \Cref{ThmFinAvs2.0}, we obtain a m.p.s. $(X,\mathcal{B},\mu,(T_g)_{g\in G})$ and some measurable function $f:X\to\mathbb{D}$ such that \begin{equation*}
\left|\int_Xfd\mu\right|^2
>
\limsup_{h\in H}\left|\int_Xf(T_hx)\overline{f(x)}d\mu(x)\right|.\qedhere
\end{equation*}
\end{enumerate}
\end{proof}

In \cite[Question 8]{BL} it is asked whether there is any implication between the notions `nice vdC set' and `set of nice recurrence'. It was also proved in \cite{BL} that if $H$ is a nice vdC set, then $H$ satisfies the following, weak version of nice recurrence: for any m.p.s. $(X,\mathcal{B},\mu,(T_g)_{g\in G})$ and for any $B\in\mathcal{B}$, we have $
\mu(B)^4\leq\limsup_{h\in H}\mu(B\cap T_hB)$. 
Both \Cref{CharactNiceVdC} and \Cref{trefdvrefds} imply that nice vdC sets are sets of nice recurrence in any amenable group.

We now recall a combinatorial version of sets of nice recurrence:

\begin{definition}
Let $G$ be a countable amenable group with a F{\o}lner sequence $F=(F_N)$. A set $H\subseteq G$ is nicely $F$-intersective if for all $E\subseteq G$ and $\varepsilon>0$ there is some $g\in R$ such that 
\begin{equation*}
\overline{d}_F(E\cap gE)>\overline{d}_F(E)^2-\varepsilon.
\end{equation*}
\end{definition}

In his blog post \cite[Section 6]{Mor}, Moreira asked whether, for a given F{\o}lner sequence $F$, every subset of the natural numbers is of nice recurrence iff it is nicely $F$-intersective. Fish and Skinner very recently established this result in \cite[Theorem 1.3]{FS} for the F{\o}lner sequence $F_N = \{1, \dots, N\}$ but left open the question of whether it holds for all F{\o}lner sequences in $\mathbb{N}$. We generalize the result to F{\o}lner sequences in amenable groups:
\begin{prop}\label{NiceRecAndInt}
For any F{\o}lner sequence $F=(F_N)$ in a countable amenable group $G$, a set $H\subseteq G\setminus\{e\}$ is of nice recurrence iff it is nicely $F$-intersective.
\end{prop}
\Cref{NiceRecAndInt} can be proved using Furstenberg's correspondence principle, once we check the following characterization of sets of nice recurrence (a general version of \cite[Proposition 3.8]{BL}):

\begin{lema}
\label{refwds54rtefdvhbdhtt}
Let $G$ be a countable group. A set $H\subseteq G\setminus\{e\}$ is of nice recurrence iff for any m.p.s. $(X,\mathcal{B},\mu,(T_g)_{g\in G})$, $B\in\mathcal{B}$ and $\varepsilon>0$ there exists $h\in H$ such that $\mu(B\cap T_hB)>\mu(B)^2-\varepsilon$.
\end{lema}

\begin{proof}
The forward implication is clear. So suppose $H\subseteq G\setminus\{e\}$ is not of nice recurrence. That means that there is a m.p.s. $(X,\mathcal{B},\mu,(T_g)_{g\in G})$, $B\in\mathcal{B}$ and $\varepsilon>0$ such that for all $h\in H$ except finitely many elements $h_1,\dots,h_j$ we have $\mu(B\cap T_hB)<\mu(B)^2-\varepsilon$.

Now consider the uniform Bernoulli shift $(Y,\mathcal{C},\nu,(S_g)_{g\in G})$, where $Y$ is the product $\{0,1,2,\dots,j\}^G$ and $S_h((x_g)_{g\in G})=(S_{gh})_{g\in G}$. Letting $e$ be the identity in $G$, the set $C:=\{(x_g)\in Y;x_e=0\textup{ and }x_{h_i}=i\textup{ 
 for }i=1,\dots,j\}\in\mathcal{C}$ satisfies that $\nu(C\cap S_gC)\in\{0,\mu(C)^2\}$ for all $g\in G\setminus\{e\}$, and in particular $\nu(C\cap S_{h_i}C)=0$ for $i=1,\dots,j$. Thus, in the product m.p.s. $(X\times Y,\mathcal{B}\times\mathcal{C},\mu\times\nu,(T_g\times S_g)_{g\in G})$, we have $(\mu\times\nu)\,((B\times C)\cap(T_h\times S_h)(B\cap C))<\nu(C)^2(\mu(B)^2-\varepsilon)$ for all $h\in H$, concluding the proof.
\end{proof}

\begin{proof}[Proof of \Cref{NiceRecAndInt}]
If $H\subseteq G\setminus\{e\}$ is not of nice recurrence, by \Cref{refwds54rtefdvhbdhtt} there is some m.p.s. $(X,\mathcal{B},\mu,(T_g)_{g\in G})$, $B\in\mathcal{B}$ and $\varepsilon>0$ such that, for all $h\in H$, $\mu(B\cap T_hB)<\mu(B)^2-\varepsilon$. Thus, by \Cref{IFC} there is some subset $E$ of $G$ such that $d_F(E)=\mu(B)$ and for all $h\in H$, $d_F(E\cap hE)=\mu(B\cap T_hB)<d_F(E)^2-\varepsilon$. So $H$ is not a nicely intersective set. The other implication can be proved similarly, using the usual Furstenberg correspondence principle instead of the inverse one.
\end{proof}

Bergelson and Moragues asked in \cite{BF} (after Remark 3.6) whether for all countable amenable groups $G$ and all F{\o}lner sequences $F$ in $G$ there is a set $E\subseteq G$ such that $\overline{d}_F(E)>0$ but for all finite $A\subseteq G$, $\overline{d}_F\left(\cup_{g\in A}g^{-1}E\right)<\frac{3}{4}$. Well, if we apply a version of \Cref{IFC} with unions instead of intersections (\Cref{IFCU}) to a m.p.s. with $T_g=\textup{Id}_X$ for all $g$ and $\mu(B)=\frac{1}{2}$, we obtain the following:
\begin{prop}\label{fwesfddsfdfs}
Let $G$ be a countably infinite amenable group with a F{\o}lner sequence $F$. Then there is a set $E\subseteq G$ such that, for all finite $\varnothing\neq A\subseteq G$,
\begin{equation*}
d_F(E)=d_F\left(\cup_{g\in A}gE\right)=\frac{1}{2}.\qed 
\end{equation*}
\end{prop}

\section{Correspondence between Cesaro and integral averages}\label{SecRrg}

The main objective of this section is proving \Cref{ThmFinAvs2.0}. It will be deduced from \Cref{ThmFinAvsGen}, a more technical version of \Cref{ThmFinAvs2.0} which also includes a finitistic criterion for the existence of sequences with given Cesaro averages. At the end of the section 
we prove \Cref{76iuyjhre5rytghf}, a general converse of the Furstenberg correspondence principle, and \Cref{IFCSemigroups}, a version of \Cref{IFC} for semigroups.

\begin{prop}\label{ThmFinAvsGen}
Let $G$ be a countably infinite amenable group with a F{\o}lner sequence $(F_N)$, and let $D\subseteq\mathbb{C}$ be compact. For each $l\in\mathbb{N}$ let $j_l\in\mathbb{N}$, $h_{l,1},\dots,h_{l,j_l}\in G$ and let $p_l:D^{j_l}\to\mathbb{C}$ be continuous. Finally, let $\gamma:\mathbb{N}\to\mathbb{C}$ be a sequence of complex numbers. The following are equivalent:
\begin{enumerate}
    \item \label{ThmFinAvsGen1} There exists a sequence $(z_g)_{g\in G}$ of elements of $D$ such that, for all $l\in\mathbb{N}$,
    \begin{equation}
    \label{54trygfdvc}
    \lim_{N\to\infty}
    \frac{1}{|F_N|}\sum_{g\in F_N}p_l(z_{h_{l,1}g},\dots,z_{h_{l,j_l}g})=\gamma(l).
    \end{equation}

    \item \label{ThmFinAvsGen2} There exists a m.p.s. $(X,\mathcal{B},\mu,(T_g)_{g\in G})$ and a measurable function $f:X\to D$ such that, for all $l\in\mathbb{N}$,
    \begin{equation}
    \label{rwedsioll}\int_Xp_l(f(T_{h_{l,1}}x),\dots,f(T_{h_{l,j_l}}x)d\mu(x)=\gamma(l).
    \end{equation}

    \item (Finitistic criterion)\label{ThmFinAvsGen3} For all $A\subseteq G$ finite and for all $L\in\mathbb{N},\delta>0$ there exist some $K\in\mathbb{N}$ and sequences $(z_{g,k})_{g\in G}$ in $D$, for $k=1,\dots,K$, such that for all $l=1,\dots,L$ we have
    \begin{equation}\label{ThmFinAvsGenEq3}
    \left|\gamma(l)-\frac{1}{K|A|}\sum_{k=1}^{K}\sum_{g\in A}p_l(z_{h_{l,1}g,k},\dots,z_{h_{l,j_l}g,k})\right|<\delta.
    \end{equation}
\end{enumerate}
\end{prop}

\begin{remark}
For the equivalence \ref{ThmFinAvsGen1}$\iff$\ref{ThmFinAvsGen3} to hold $D$ need not be compact, and $p_l:D^{j_l}\to\mathbb{C}$ can be any bounded function (not necessarily continuous). Indeed, all we will use in the proof of \ref{ThmFinAvsGen3}$\implies$\ref{ThmFinAvsGen1} is that the functions $p_l$ are bounded, and it is not hard to show \ref{ThmFinAvsGen1}$\implies$\ref{ThmFinAvsGen3} if the functions $p_l$ are bounded: suppose that $(z_g)$ satisfies \ref{ThmFinAvsGen1} for some F{\o}lner sequence $(F_N)$. Then one can check that given $L\in\mathbb{N},\delta>0$, for big enough $N\in\mathbb{N}$, the constant $K=|F_N|$ and the sequences $(z_{g,k})_g$ given by $z_{g,k}=z_{gk}$, for $k\in F_N$, satisfy \Cref{ThmFinAvsGen3}.
\end{remark}

\begin{proof}[Proof of \Cref{ThmFinAvs2.0} from \Cref{ThmFinAvsGen}]
Let $G,(F_N),(X,\mathcal{B},\mu,(T_g)_{g\in G}),$ $D$ and $f:X\to D$ be as in \Cref{ThmFinAvs2.0}.

Now, for each $k\in\mathbb{N}$, the set $C(D^k):=\{p:D^k\to\mathbb{C};p\text{ continuous}\}$ is separable in the supremum norm. So we can consider for each $l\in\mathbb{N}$ elements $h_{l,1},\dots,h_{l,j_l}$ and functions $p_l\in C(D^{j_l})$ such that for any $h_1,\dots,h_j\in G$, for any $p\in C(D^j)$ continuous and for any $\varepsilon>0$ there exists $l\in\mathbb{N}$ such that $j_l=j$, $(h_{l,1},\dots,h_{l,j_l})=(h_1,\dots,h_j)$ and $\|p-p_l\|_\infty<\varepsilon$.

If we now apply \Cref{ThmFinAvsGen} to the sequence 
\begin{equation*}
\gamma(l)=\int_Xp_l(f(T_{h_{l,1}}x),\dots,f(T_{h_{l,j_l}}x)d\mu(x),
\end{equation*}
we obtain a sequence $(z_g)_{g\in G}$ of elements of $D$ such that, for all $l\in\mathbb{N}$,
\begin{equation}\label{65ythrfg45trdfg}
\lim_{N\to\infty}\frac{1}{|F_N|}\sum_{g\in F_N}p_l(z_{h_{l,1}g},\dots,z_{h_{l,j_l}g})=\int_Xp_l(f(T_{h_{l,1}}x),\dots,f(T_{h_{l,j_l}}x)d\mu(x).
\end{equation}
This implies that \Cref{r4fdfrdedfdfsdsf} holds for any $h_1,\dots,h_j\in G$ and $p\in C(D^k)$, so we are done proving the first implication.

The converse implication of \Cref{ThmFinAvs2.0} follows from standard
proofs of the the Furstenberg correspondence principle, see e.g. \cite[Pages 922-923]{BF}. We give a short proof using \Cref{ThmFinAvsGen}.
Given a sequence $(z_g)_{g\in G}$ in $D$, first choose a F{\o}lner subsequence $(F_N')_{n\in\mathbb{N}}$ such that the following limit exists for all $l$.
\begin{equation*}
\gamma(l)=\lim_{N\to\infty}\frac{1}{|F_N'|}\sum_{g\in F_N'}p_l(z_{h_{l,1}g},\dots,z_{h_{l,j_l}g}).
\end{equation*}
So, by \Cref{ThmFinAvsGen} there is a m.p.s. $(X,\mathcal{B},\mu,(T_g)_{g\in G})$ and a measurable function $f:X\to D$ such that, for all $l\in\mathbb{N}$, \Cref{rwedsioll} holds. In particular, \Cref{r4fdfrdedfdfsdsf} holds if the left hand side limit exists (as the limit along the sequence $(F_N)$ is the same as along the sequence $(F_N')$), as we wanted.
\end{proof}

We will now prove \Cref{ThmFinAvsGen}. We will show \ref{ThmFinAvsGen1}$\implies$\ref{ThmFinAvsGen2}$\implies$\ref{ThmFinAvsGen3}$\implies$\ref{ThmFinAvsGen1}. The complicated implication is \ref{ThmFinAvsGen3}$\implies$\ref{ThmFinAvsGen1}. During the rest of the proof we fix an amenable group $G$, a F{\o}lner sequence $(F_N)$ and a compact set $D\subseteq\mathbb{C}$.

\begin{proof}[Proof of \ref{ThmFinAvsGen1}$\implies$\ref{ThmFinAvsGen2}]
Let $\s{z}=(z_a)_{a\in G}$ be as in \ref{ThmFinAvsGen1}. Consider the (compact, metrizable) product space $D^G$ of sequences $(z_g)_{g\in G}$ in $D$, with the Borel $\sigma$-algebra. We have an action $(R_g)_{g\in G}$ of $G$ on $D^G$ by $R_g((z_a)_{a\in G})=(z_{ag})_{a\in G}$.

For each $N$ let $\nu_N$ be the average of Dirac measures $\frac{1}{|F_N|}\sum_{g\in F_N}\delta_{R_g\s{z}}$. Let $\nu$ be the weak limit of some subsequence $(\nu_{N_m})_m$; then $\nu$ is invariant by $R_h$ for all $h\in G$, because $\lim_N\frac{|F_N\Delta hF_N|}{|F_N|}=0$ and for all $N\in\mathbb{N}$,
\begin{equation*}
\nu_N-(R_h)_*\nu_N
=
\frac{1}{|F_N|}\left(\sum_{g\in F_N\setminus hF_N}\delta_{R_{g}\s{z}}-\sum_{g\in hF_N\setminus F_N}\delta_{R_{g}\s{z}}\right).
\end{equation*}
Finally, we define $f:D^G\to\mathbb{C}$ by $f((z_a)_{a\in G})=z_{e}$ (where $e\in G$ is the identity). Then for all $l\in\mathbb{N}$ we have
\begin{align*} &\int_{D^G}p_l(f(R_{h_{l,1}}x),\dots,f(R_{h_{l,j_l}}x))d\nu(x)\\
=&\lim_m
\int_{D^G}p_l(f(R_{h_{l,1}}x),\dots,f(R_{h_{l,j_l}}x))d\nu_{N_m}(x)\\
=&\lim_m\frac{1}{|F_{N_m}|}\sum_{g\in F_{N_m}}
p_l(f(R_{h_{l,1}}(R_g\s{z})),\dots,f(R_{h_{l,j_l}}(R_g(\s{z})))\\
=&\lim_m\frac{1}{|F_{N_m}|}\sum_{g\in F_{N_m}}p_l(z_{h_{l,1}g},\dots,z_{h_{l,j_l}g})=\gamma(l).\qedhere
\end{align*}
\end{proof}

\begin{proof}[Proof of \ref{ThmFinAvsGen2}$\implies$\ref{ThmFinAvsGen3}]
We let $(X,\mathcal{B},\mu,(T_g)_{g\in G})$ and $f$ be as in \ref{ThmFinAvsGen2}, and we define the m.p.s. $(D^G,\mathcal{B}(D^G),(R_g)_{g\in G})$ as in the proof of \ref{ThmFinAvsGen1}$\implies$\ref{ThmFinAvsGen2}. 

Consider the function $\Phi:X\to D^G$ given by $\Phi(x)=(f(T_ax))_{a\in G}$ and let $\nu$ be the measure $\Phi_*\mu$ in $D^G$. As a Borel probability measure on a compact metric space, $\nu$ can be weakly approximated by a sequence of finitely supported measures $(\nu_N)_{n\in\mathbb{N}}$, which we may assume are of the form
\begin{equation*}
\nu_N=\frac{1}{|K_N|}\sum_{k\in K_N}\delta_{(z_{a,k})_{a\in G}}.
\end{equation*}
For some finite index sets $K_N$ and some sequences $(z_{a,k})_{a\in G}$, for $k\in K_N$. Now, the fact that $\nu_N\to\nu$ weakly means that for all $l\in\mathbb{N}$ and $A\subseteq G$ finite we have
\begin{align*}
&\lim_N\frac{1}{|K_N|\cdot|A|}\sum_{k\in K_N}\sum_{g\in A}p_l(z_{h_{l,1}g,k},\dots,z_{h_{l,j_l}g,k})\\
=&
\lim_N\frac{1}{|A|}\sum_{g\in A}\int_{D^G}
p_l(z_{h_{l,1}g},\dots,z_{h_{l,j_l}g})d\nu_N((z_a))\\
=&
\frac{1}{|A|}\sum_{g\in A}\int_{D^G}
p_l(z_{h_{l,1}g},\dots,z_{h_{l,j_l}g})d\nu((z_a))\\
=&
\frac{1}{|A|}\sum_{g\in A}\int_{X}
p_l(f(T_{h_{l,1}g}x),\dots,f(T_{h_{l,j_l}g}x))
d\mu(x)\\
=&\frac{1}{|A|}\sum_{g\in A}\int_{X}
p_l(f(T_{h_{l,1}}x),\dots,f(T_{h_{l,j_l}}x))d\mu(x)=\frac{1}{|A|}\sum_{g\in A}\gamma(l)=\gamma(l).
\end{align*}
So for any $L\in\mathbb{N}$ and $\delta>0$, \Cref{ThmFinAvsGenEq3} will be satisfied for big enough $N$.
\end{proof}

In order to prove \ref{ThmFinAvsGen3}$\implies$\ref{ThmFinAvsGen1}, we will first need some lemmas about averages of sequences of complex numbers.

\begin{prop}\label{AvgsSymDif}
Let $M>0$ and $\delta\in(0,1)$. For any finite sets $E'\subseteq E\subseteq G$ such that $\frac{|E\setminus E'|}{|E|}<\delta$ and any complex numbers $(z_g)_{g\in G}$ with $|z_g|\leq M\;\forall g\in G$, we have
\begin{equation*}
\left|\frac{1}{|E|}\sum_{g\in E}z_g-\frac{1}{|E'|}\sum_{g\in E'}z_g\right|<2M\delta.
\end{equation*}
\end{prop}

\begin{proof}
The left hand side above is equal to 
\begin{align*}
\left|\frac{|E'|-|E|}{|E|\cdot|E'|}\sum_{g\in E'}z_g+\frac{1}{|E|}\sum_{g\in E\setminus E'}z_g
\right|
&\leq\frac{|E|-|E'|}{|E|\cdot|E'|}\cdot M|E'|
+\frac{1}{|E|}\cdot M|E\setminus E'|\\
&< M\delta+M\delta.\qedhere
\end{align*}
\end{proof}

\begin{definition}
Let $S,T$ be subsets of a group $G$. We denote 
\begin{equation*}
\partial_ST:=\{g\in G;Sg\cap T\neq\varnothing\text{ and }Sg\cap 
(G\setminus T)\neq\varnothing\}
=
\left(\cup_{s\in S}s^{-1}T\right)\setminus\left(\cap_{s\in S}s^{-1}T\right).
\end{equation*}
\end{definition}

\begin{remark}\label{trtrefd}
Note that if $(F_N)$ is a F{\o}lner sequence in $G$ and $S$ is finite, then $\lim_N\frac{|\partial_SF_N|}{|F_N|}=0$, because $\lim_N\frac{|s^{-1}F_N\Delta F_N|}{|F_N|}=0$ for all $s\in G$. Also note that if $c\in G$, then $\partial_S(Tc)=(\partial_ST)c$. Finally, if $e\in S$, then for all $g\in T\setminus\partial_ST$ we have $Sg\subseteq T$.
\end{remark}

The following lemma is a crucial part of the proof of \Cref{ThmFinAvsGen}. 
Intuitively, it says that if you have a finite set of bounded sequences $f_i:T_i\to\mathbb{C}$ defined in finite subsets $T_1,\dots,T_k$ of a group $G$, you can reassemble them into a sequence $f:A\to\mathbb{C}$ defined in a bigger finite set $A\subseteq G$ which is a union of right translates of the sets $T_i$, 
and the average value of $f$ will approximately be the weighted average of the average values of the $f_i$.

\begin{lema}\label{AvgsSubsets'}
Let $G$ be a group, $\delta,M>0,K\in\mathbb{N}$ and $c_1,\dots,c_K\in G$. Let $T_1,\dots,T_K,A$ be finite subsets of $G$ such that $T_1c_1,\dots,T_Kc_K$ are pairwise disjoint, contained in $A$ and $\frac{\sum_{k=1}^K|T_k|}{|A|}>1-\frac{\delta}{M}$. For $k=1,\dots,K$ let $(z_{g,k})_{g\in G}$ be sequences of complex numbers in $D$, and consider a sequence $(z_g)_{g\in G}$ such that 
\begin{equation*}
z_g=z_{gc_{k}^{-1},k}\text{ for all }g\in T_kc_k,
\end{equation*}
Finally, let $S=\{h_1,\dots,h_j\}\subseteq G$ be finite with $e\in S$ and let $p:D^j\to\mathbb{C}$ satisfy $\|p\|_\infty\leq M$. If we have $\frac{|\partial_ST_k|}{|T_k|}\leq\frac{\delta}{M}$ for all $k$, then
\begin{equation*}\label{43rewfsdrewsd'}
\left|\frac{1}{|A|}\sum_{g\in A}
p(z_{h_1g},\dots,z_{h_jg})
-
\sum_{k=1}^K\frac{|T_k|}{\sum_{i=1}^K|T_i|}\left(\frac{1}{|T_k|}\sum_{g\in T_k}p(z_{h_1g,k},\dots,z_{h_jg,k})\right)
\right|\leq4\delta.
\end{equation*}
\end{lema}

\begin{proof}
Considering sequences of complex numbers as functions $f:G\to\mathbb{C};a\mapsto f(a)$, which we write during this proof as $(f(a))_a$, we will denote
\begin{equation*}
p(f):=p(f(h_1),\dots,f(h_j)).
\end{equation*}

Note that, given $g\in G$, we will have $p((z_{a gc_k^{-1},k})_a)=p((z_{ag})_a)$ whenever $z_{agc_k^{-1},k}=z_{ag}$ for all $a\in S$, which happens when $ag\in T_kc_k$ for all $a\in S$, that is, when $g\in T_kc_k\setminus\partial_ST_kc_k$. Thus, for each $k=1,\dots,K$ we have
\begin{align*}
&\left|
\frac{1}{|T_k|}\sum_{g\in T_k}
p((z_{ag,k})_a)
-\frac{1}{|T_k|}\sum_{g\in T_kc_k}
p((z_{ag})_a)
\right|
\\=&
\left|
\frac{1}{|T_k|}\sum_{g\in T_kc_k}
p((z_{agc_k^{-1},k})_a)
-\frac{1}{|T_k|}\sum_{g\in T_kc_k}
p((z_{ag})_a)
\right|
\\
=&\left|
\frac{1}{|T_k|}\sum_{g\in T_kc_k\cap\partial_ST_kc_k}\left(p((z_{agc_k^{-1},k})_a)-p((z_{ag})_a
)\right)
\right|
\leq2\delta,
\end{align*}
the last inequality being because $\frac{|T_kc_k\cap \partial_ST_kc_k|}{|T_k|}<\frac{\delta}{M}$ and the summands have norm $\leq2M$. Taking weighted averages over $k=1,\dots,K$, we obtain
\begin{multline*}
\left|
\sum_{k=1}^K\frac{|T_k|}{\sum_{i=1}^K|T_i|}\cdot\left(\frac{1}{|T_k|}\sum_{g\in T_k}p((z_{ag,k})_a)\right)\right.
\\\left.-
\sum_{k=1}^K\frac{|T_k|}{\sum_{i=1}^K|T_i|}\cdot\left(\frac{1}{|T_k|}\sum_{g\in T_kc_k}p((z_{ag})_a)\right)
\right|\leq2\delta.
\end{multline*}
But applying \Cref{AvgsSymDif} with $E=A$, $E'=\cup_iT_ic_i$ and the sequence $g\mapsto p((z_{ag})_a)$ we also have
\begin{equation*}
\left|\frac{1}{|A|}\sum_{g\in A}p((z_{ag})_a)
-
\sum_{k=1}^K\frac{|T_k|}{\sum_{i=1}^K|T_i|}\cdot\left(\frac{1}{|T_k|}\sum_{g\in T_kc_k}p((z_{ag})_a)\right)
\right|\leq2\delta.
\end{equation*}
So by the triangle inequality, we are done.
\end{proof}

The only non elementary fact that we will need in the proof of \ref{ThmFinAvsGen3}$\implies$\ref{ThmFinAvsGen1} is \cite[Theorem 5.2]{DHZ}; in order to state it we first recall some definitions:
\begin{definition}[cf. {\cite[Defs 3.1,3.2]{DHZ}}]
A \textit{tiling} $\mathcal{T}$ of a group $G$ consists of two objects:
\begin{itemize}
    \item A finite family $\mathcal{S}(\mathcal{T})$ (the \textit{shapes}) of finite subsets of $G$ containing the identity $e$.
    \item A finite collection $C(\mathcal{T})=\{C(S);S\in \mathcal{S}(\mathcal{T})\}$ of subsets of $G$, the \textit{center sets}, such that the family of right translates of form $Sc$ with $S\in\mathcal{S}(\mathcal{T})$ and $c\in C(S)$ (such sets $Sc$ are the \textit{tiles} of $\mathcal{T}$)
    form a partition of $G$.
\end{itemize}

\end{definition}

\begin{definition}[cf. {\cite[Page 17]{DHZ}}]
We say a sequence of tilings $(\mathcal{T}_k)_{k\in\mathbb{N}}$ of a group $G$ is \textit{congruent} if every tile of $\mathcal{T}_{k+1}$ equals a union of tiles of $\mathcal{T}_k$.
\end{definition}

\begin{definition}
If $A$ and $B$ are finite subsets of a group $G$, we say that $A$ is $(B,\varepsilon)$-invariant if $\frac{|\partial_BA|}{|A|}<\varepsilon$.\footnote{This is not the definition of $(B,\varepsilon)$-invariant used in \cite{DHZ}, but it will be more convenient for our purposes.}
\end{definition}

We will need the following weak version of \cite[Theorem 5.2]{DHZ}:
\begin{theorem}[cf. {\cite[Theorem 5.2]{DHZ}}]\label{DHZThm5.2}
For any countable amenable group $G$ there exists a congruent sequence of tilings $(\mathcal{T}_k)_{k\in\mathbb{N}}$ such that, for every $A\subseteq G$ finite and $\delta>0$, all the tiles of $\mathcal{T}_k$ are $(A,\delta)$-invariant for big enough $k$.
\end{theorem}

The following notation will be useful.
\begin{definition}\label{DefBdTiling}
For any tiling of a group $G$ and any finite set $B\subseteq G$, we let $\partial_{\mathcal{T}}B$ denote the union of all tiles of $\mathcal{T}$ which intersect both $B$ and $G\setminus B$.
\end{definition}

\begin{remark}
Let $G$ be a countable amenable group with a tiling $\mathcal{T}$ and a F{\o}lner sequence $(F_N)$. As $\partial_{\mathcal{T}}F_N\subseteq\cup_{S\in\mathcal{S}(\mathcal{T})}\partial_SF_N$, \Cref{trtrefd} implies that 
\begin{equation*}
\lim_{N\to\infty}\frac{|\partial_{\mathcal{T}}F_N|}{|F_N|}=0.
\end{equation*}
\end{remark}

\begin{lema}\label{5y4trgfdvtrgfdv}
Let $G$ be a countably infinite amenable group with a F{\o}lner sequence $(F_N)$, and let $(\mathcal{T}_k)_{k\in\mathbb{N}}$ be a congruent sequence of tilings of $G$. Then there is a partition $\mathcal{P}$ of $G$ into tiles of the tilings $\mathcal{T}_k$ such that
\begin{itemize}
    \item For each $k\in\mathbb{N}$, $\mathcal{P}$ contains only finitely many tiles of $\mathcal{T}_k$.
    \item If for each $N\in\mathbb{N}$ we let $A_N=\bigcup\{T\in\mathcal{P};T\subseteq F_N\}\subseteq F_N$, then we have
    \begin{equation*}
    \lim_{N\to\infty}\frac{|A_N|}{|F_N|}=1.
    \end{equation*}
\end{itemize}
\end{lema}

\begin{proof}
We can assume that $\cup_NF_N=G$, adding some F{\o}lner sets to the sequence if necessary. In the following, for each finite set $B\subseteq G$ and $k\in\mathbb{N}$ we will denote $\partial_kB:=\cup_{j=1}^k\partial_{\mathcal{T}_j}B$, so that for all $k\in\mathbb{N}$,
\begin{equation*}
\lim_{N\to\infty}\frac{|\partial_kF_N|}{|F_N|}=0.
\end{equation*}
Now for each $k\in\mathbb{N}$ let $N_k$ be a big enough number that $\frac{|\partial_{k+1}F_N|}{|F_N|}\leq\frac{1}{k+1}$ for all $N\geq N_k$.

Let $D_0=\varnothing$ and for each $k\in\mathbb{N}$ let $D_k\subseteq G$ be the union of all tiles of $\mathcal{T}_{k+1}$ intersecting some element of $\bigcup_{N=1}^{N_k}F_N$. Thus, $D_k\subseteq D_{k+1}$ for all $k$, $G=\cup_kD_k$ and $D_{k}\setminus D_{k-1}$ is a union of tiles of $\mathcal{T}_k$. We define $\mathcal{P}$ to be the partition of $G$ formed by all tiles of $\mathcal{T}_k$ contained in $D_{k}\setminus D_{k-1}$, for all $k\in\mathbb{N}$.

Now, fix $N$ and let $k$ be the smallest natural number such that $N\leq N_k$ (note that $k\to\infty$ when $N\to\infty$). 
Note that all tiles $T\in\mathcal{P}$ intersecting $F_N$ must be in $\mathcal{T}_j$ for some $j\leq k$. Thus, the set $A_N$ of all tiles of $\mathcal{P}$ which are contained in $F_N$ must contain $|F_N\setminus\partial_kF_N|$. But we have $N>N_{k-1}$, so $\frac{|\partial_kF_N|}{|F_N|}\leq\frac{1}{k}$. So $\frac{|A_N|}{|F_N|}\geq1-\frac{1}{k}$, and we are done.\qedhere
\end{proof}

\begin{proof}[Proof of \Cref{ThmFinAvsGen3}$\implies$\Cref{ThmFinAvsGen1}]
Let $S_l=\{h_{l,1},\dots,h_{l,j_l}\}$ and $M_l\geq\|p_l\|_\infty$ for all $l\in\mathbb{N}$ (we may assume $e\in S_l$ and $M_{l+1}\geq M_l$ for all $l$). We prove first that \ref{ThmFinAvsGen3} implies the following:
\begin{enumerate}[label=\arabic*'.]
\setcounter{enumi}{2}
    \item \label{ThmFinAvsGen3'}Let $\delta>0$ and $L\in\mathbb{N}$. Then for any sufficiently left-invariant\footnote{By this we mean that
    there is a finite set $A\subseteq G$ and some $\varepsilon>0$ such that the property stated below is satisfied for all $(A,\varepsilon)$-invariant sets.} subset $B$ of $G$ there exists a sequence $(w_g)_{g\in G}$ in $D$ such that, for all $l=1,\dots,L$,
    \begin{equation*}
    \left|\gamma(l)-\frac{1}{|B|}\sum_{g\in B}
    p_l(w_{h_{l,1}g},\dots,w_{h_{l,j_l}g})\right|<\delta.
    \end{equation*}
\end{enumerate}

In order to prove \ref{ThmFinAvsGen3'} from \ref{ThmFinAvsGen3}, fix $\delta,L$ and consider a tiling $\mathcal{T}$ of $G$ such that $\frac{|\partial_{S_l}S|}{|S|}<\frac{\delta}{10M_L}$ for all $l=1,\dots,L$ and $S\in\mathcal{S}(\mathcal{T})$. For each $S\in\mathcal{S}(\mathcal{T})$ there is by \ref{ThmFinAvsGen3} some $K_S\in\mathbb{N}$ and sequences $(z_{S,g,k})_{g\in G}$ in $D$ ($k=1,\dots,K_S$) satisfying 
\begin{equation}\label{453terfdv}
\left|\gamma(l)-\frac{1}{K_S|S|}\sum_{k=1}^{K_S}\sum_{g\in S}p_l(z_{S,h_{l,1}g,k},\dots,z_{S,h_{l,j_l}g,k})\right|<\frac{\delta}{5}\text{ for }l=1,\dots,L.
\end{equation}
Then any finite subset $B$ of $G$ such that $\frac{|\partial_TB|}{|B|}<\frac{\delta}{10M_L}$ and \begin{equation}\label{sizeB}
|B|\geq\frac{10M_L}{\delta}\sum_{S\in\mathcal{S}(\mathcal{T})}|S|K_S
\end{equation} will satisfy \ref{ThmFinAvsGen3'}; to see why, first note that the union $B_0$ of all 
tiles of $\mathcal{T}$ contained in $B$ satisfies $\frac{|B_0|}{|B|}>1-\frac{\delta}{10M_L}$; now obtain a set $B_1\subseteq B_0$ by removing finitely many tiles from $B_0$ in such a way that, for each $S\in\mathcal{S}(\mathcal{T})$, the number of tiles of shape $S$ contained in $B_1$ is a multiple of $K_S$. Note that, due to \Cref{sizeB}, this can be done in such a way that $\frac{|B_1|}{|B|}\geq1-\frac{\delta}{5M_L}$.

So letting $\mathcal{P}$ be the set of tiles of $\mathcal{T}$ contained in $B_1$, we can define a function $f:\mathcal{P}\to\mathbb{N}$ that associates for each translate $Sc$ some value in $\{1,2,\dots,K_S\}$, and such that for each $S$, the number of right-translates of $S$ having value $k$ is the same for all $k=1,2,\dots,K_S$. Finally, define a sequence $(w_g)_{g\in G}$ by
\begin{equation*}
w_g=z_{S,gc^{-1},f(Sc)}\text{ if }g\in Sc\text{ and }Sc\text{ is a tile of }\mathcal{T}\text{ contained in }B_1.
\end{equation*}
The rest of values of $w_g$ are not important, we can let them be some fixed value of $D$.
Then, by \Cref{AvgsSubsets'} applied to the set $B$ and the tiles of $\mathcal{T}$ contained in $B_1$, we obtain for all $l=1,\dots,L$ that
\begin{multline*}
\left|\frac{1}{|B|}\sum_{g\in B}
p(w_{h_{l,1}g},\dots,w_{h_{l,j_l}})\right.
\\\left.-
\sum_{Sc\in\mathcal{P}}\frac{|Sc|}{|B_1|}\cdot\left(\frac{1}{|S|}\sum_{g\in S}
p(z_{S,h_{l,1}g,f(Sc)},\dots,z_{S,h_{l,j_l}g,f(Sc)})\right)
\right|\leq4\left(\frac{\delta}{5}\right).
\end{multline*}
But as the sequences $(z_{S,g,k})_{g\in G}$ in $D$ ($k=1,\dots,K_S$) satisfy \Cref{453terfdv}, we have for all $l=1,\dots,L$ that
\begin{equation*}
\left|\gamma(l)
-
\sum_{Sc\in\mathcal{P}}\frac{|Sc|}{|B_1|}\cdot\left(\frac{1}{|S|}\sum_{g\in S}
p(z_{S,h_{l,1}g,f(Sc)},\dots,z_{S,h_{l,j_l}g,f(Sc)})\right)
\right|\leq\frac{\delta}{5},
\end{equation*}
so we are done proving \ref{ThmFinAvsGen3'} by the triangle inequality.

Now we use \ref{ThmFinAvsGen3'} to prove \ref{ThmFinAvsGen1}. By \ref{ThmFinAvsGen3'} and \Cref{DHZThm5.2}, there is a congruent sequence of tilings $(\mathcal{T}_L)_{L\in\mathbb{N}}$ of $G$ (we may also assume $\mathcal{S}(\mathcal{T}_L)\cap \mathcal{S}(\mathcal{T}_{L'})=\varnothing$ if $L\neq L'$) such that
\begin{enumerate}
    \item For all $S\in\mathcal{S}(\mathcal{T}_L)$ and for $l=1,\dots,L$ we have $\frac{|\partial_{S_l}S|}{|S|}<\frac{1}{L}$.
    \item For each $S\in\mathcal{S}(\mathcal{T}_L)$ there is a sequence $(w_{S,g})_{g\in G}$ such that for all $l=1,\dots,L$,
    \begin{equation}\label{ayay2}
    \left|\gamma(l)-\frac{1}{|S|}\sum_{g\in S}p_l(w_{S,h_{l,1}g},\dots,w_{S,h_{l,j_l}g})\right|<\frac{1}{L}.
    \end{equation}
\end{enumerate}
Now, letting $(F_N)$ be the F{\o}lner sequence of \Cref{ThmFinAvsGen1}, we let $\mathcal{P}$ and $(A_N)_{n\in\mathbb{N}}$ be as in \Cref{5y4trgfdvtrgfdv}, with $\mathcal{P}_L\subseteq\mathcal{P}$ ($L=1,2,\dots$) being finite sets of tiles of $\mathcal{T}_L$ such that $\mathcal{P}=\sqcup_L\mathcal{P}_L$. We define the sequence $(z_g)_{g\in G}$ by
\begin{equation*}
z_g=w_{S,gc^{-1}}\text{ if }g\in Sc,\text{ where }S\in\mathcal{S}(\mathcal{T}_L)\text{ for some $L$ and }Sc\in\mathcal{P}_L.
\end{equation*}
All that is left is proving that $(z_g)_g$ satisfies \ref{ThmFinAvsGen1} for all $l\in\mathbb{N}$. So let $\varepsilon>0$ and $l\in\mathbb{N}$. 
Note that for big enough $N$ we have $\frac{|A_N|}{|F_N|}>1-\frac{\varepsilon}{10M_l}$. 
Fix some natural number $L>\frac{10}{\varepsilon}$ (also suppose $L\geq l$). 
Letting $A_N'$ be obtained from $A_N$ by removing the (finitely many) tiles which are in $\mathcal{P}_i$ for some $i=1,\dots,L$, then for big enough $N$ we have $\frac{|A_N'|}{|F_N|}>1-\frac{\varepsilon}{5M_l}$. 
Thus, applying \Cref{AvgsSubsets'} to the set $F_N$ and to all the tiles contained in $A_N'$, and for each $S\in\mathcal{S}(\mathcal{T}_L)$ letting $n_S$ be the number of tiles of shape $S$ in $\mathcal{P}_L$ which are contained in $A_N'$, we obtain
\begin{multline}\label{6yrtgdfy6trgfd'}
\left|\frac{1}{|F_N|}\sum_{g\in F_N}
p_l(z_{h_{l,1}g},\dots,z_{h_{l,j_l}g})\right.
\\-\left.
\sum_{L\in\mathbb{N}}\sum_{S\in\mathcal{S}(\mathcal{T}_L)}\frac{n_S|S|}{|A_N'|}\cdot\left(\frac{1}{|S|}\sum_{g\in S}
p_l(w_{S,h_{l,1}g,L},\dots,w_{S,h_{l,j_l}g,L})\right)
\right|\leq4\frac{\varepsilon}{5}.
\end{multline}
However, the double sum in \Cref{6yrtgdfy6trgfd'} is at distance $\leq\frac{\varepsilon}{10}$ of $\gamma(l)$ 
(this follows from taking an affine combination of  \Cref{ayay2} applied to the tiles $S$ of $A_N'$, with constant $\frac{1}{L}<\frac{\varepsilon}{10}$). So by the triangle inequality, for big enough $N$ we have
\begin{equation*}
\left|\gamma(l)-\frac{1}{|F_N|}\sum_{g\in F_N}
p_l(z_{h_{l,1}g},\dots,z_{h_{l,j_l}g})
\right|<\varepsilon.
\end{equation*}
As $\varepsilon$ is arbitrary, we are done.
\end{proof}

As promised in the introduction, we now prove a converse to Furstenberg's correspondence principle.
\Cref{76iuyjhre5rytghf} is a slight generalization of \Cref{IFC} in which we also allow intersections with complements of the sets $g_iA$. In the following, for any $A\subseteq G$ and $B\subseteq X$, we denote $A^1=A,A^0=G\setminus A$, $B^1=B$ and $B^0=X\setminus B$.

\begin{theorem}\label{76iuyjhre5rytghf}
Let $G$ be a countably infinite amenable group with a F{\o}lner sequence $(F_N)$. For every m.p.s. $(X,\mathcal{B},\mu,(T_g)_{g\in G})$ and every $B\in\mathcal{B}$ there exists a subset $A\subseteq G$ such that, for all $n\in\mathbb{N}$, $\varepsilon_1,\dots,\varepsilon_n\in\{0,1\}$ and $g_1,\dots,g_n\in G$ we have
\begin{equation}\label{EqStrongFurst}
d_F\left(\bigcap_{i=1}^n(g_iA)^{\varepsilon_i}\right)
=
\mu\left(\bigcap_{i=1}^k(T_{g_i}B)^{\varepsilon_i}\right).
\end{equation}
Reciprocally, for any subset $A\subseteq G$ there is a m.p.s. $(X,\mathcal{B},\mu,T)$ and $B\in\mathcal{B}$ satisfying \Cref{EqStrongFurst} for all $n,\varepsilon_1,\dots,\varepsilon_n,g_1,\dots,g_n$ as above such that the density in \Cref{EqStrongFurst} exists.
\end{theorem}

\begin{proof}
Let $p_0,p_1:\{0,1\}\to\mathbb{C}$ be given by $p_0(x)=1-x$ and $p_1(x)=x$.
Note that for all $n,\varepsilon_1,\dots,\varepsilon_n,g_1,\dots,g_n$ we have
\begin{gather*}
\mu\left(\bigcap_{i=1}^k(T_{g_i}B)^{\varepsilon_i}\right)=
\int_X
\prod_{i=1}^kp_{\varepsilon_i}\left(\chi_B\left(T_{g_i^{-1}}x\right)\right)
d\mu.
\end{gather*}
So by \Cref{ThmFinAvs2.0} applied to the polynomials of the form $p(x_1,\dots,x_n)=\prod_{i=1}^kp_{\varepsilon_i}(x_i)$ (for all $n,\varepsilon_i,g_i$) and with $D=\{0,1\}$, there exists some characteristic function $\chi_A:G\to\{0,1\}$ such that for all $n,(\varepsilon_i)_{i=1}^n$ and $(g_i)_{i=1}^n$ we have
\begin{gather*}
\mu\left(\bigcap_{i=1}^k(T_{g_i}B)^{\varepsilon_i}\right)
=\lim_{N\to\infty}\frac{1}{|F_N|}
\sum_{g\in F_N}
\prod_{i=1}^kp_{\varepsilon_i}\left(\chi_A(g_i^{-1}g)\right)
=d_F\left(\bigcap_{i=1}^k(g_iA)^{\varepsilon_i}\right).
\end{gather*}
The other implication is proved similarly, applying \Cref{ThmFinAvs2.0} in the other direction.
\end{proof}

\Cref{76iuyjhre5rytghf} is a statement about densities of any sets in the algebra $\mathcal{A}\subseteq\mathcal{P}(G)$ generated by the family $\{gA;g\in G\}$. For example, using that for any $g_1,\dots,g_n\in G$ the set $\cup_{i=1}^ng_iA$ is the union for all $\varepsilon_1,\dots,\varepsilon_n$ not all equal to $1$ of $\cap_{i=1}^n(g_iA)^{\varepsilon_i}$, we obtain a version of Furstenberg's correspondence principle with unions:

\begin{theorem}[Inverse Furstenberg correspondence principle with unions]
\label{IFCU}
Let $G$ be a countably infinite amenable group with a F{\o}lner sequence $(F_N)$. For every m.p.s. $(X,\mathcal{B},\mu,(T_g)_{g\in G})$ and every $B\in\mathcal{B}$ there exists a subset $A\subseteq G$ such that for all $k\in\mathbb{N}$ and $h_1,\dots,h_k\in G$ we have
\begin{equation*}
d_F\left(h_1A\cup\dots\cup h_kA\right)
=
\mu\left(T_{h_1}B\cup\dots\cup T_{h_k}B\right).\qed
\end{equation*}
\end{theorem}

\begin{remark}
It is possible to give a version of \Cref{IFC} that involves translates of several sets $B_1,\dots,B_l\subseteq X$, as in \cite[Definition A.3]{BF2}, or even countably many sets $(B_k)_{k\in\mathbb{N}}$. We will not do so as it falls outside the scope of this article and it would involve proving a modified version of \Cref{ThmFinAvsGen}.
\end{remark}

We include a version of \Cref{IFC} for semigroups. 

\begin{definition}
\label{DefSemigroupProps}
Let $(S,\cdot)$ be a semigroup. For $A\subseteq S$ and $s\in S$, we let $s^{-1}A=\{x\in S;sx\in A\}$. We say $S$ is (left) amenable\footnote{This definition is known to be equivalent to \Cref{DefAmenableGroup} for countable groups.} if there is a finitely additive probability measure $\mu:\mathcal{P}(S)\to[0,1]$ such that $\mu(s^{-1}A)=\mu(A)$ for all $s\in S,A\subseteq S$. We say that $S$ is cancellative when for all $a,b,c\in S$, $ab=ac$ implies $b=c$ and $ba=ca$ implies $b=c$. 
A F{\o}lner sequence in $S$ is a sequence $(F_N)_{N\in\mathbb{N}}$ of finite subsets of $S$ such that, for all $s\in S$, $\lim_{N\to\infty}\frac{|F_N\Delta sF_N|}{|F_N|}=0$.
\end{definition}

For the rest of the section we fix a countable, (left) amenable, (two-sided) cancellative semigroup $S$. In particular, $S$ is \textit{left-reversible}, that is, $aS\cap bS\neq\varnothing$ for all $a,b\in S$ (in fact, $\mu(aS\cap bS)=1$ for any left-invariant mean $\mu$ in $S$). An argument dating back to \cite{Ore} implies that $S$ can be imbedded into a group $G$, the \textit{group of right quotients} of $S$, such that $G=\{st^{-1};s,t\in S\}$, see \cite[Theorem 1.23]{Cli}. 

We will use the fact that any $S$-m.p.s. can be extended to a $G$-m.p.s. Extensions of measure-theoretic and topological semigroup actions to groups have quite recently been studied in \cite{FJB,Ec,BBD1,BBD2}.

\begin{definition}
We say a measurable space $(X,\mathcal{B})$ is \textit{standard Borel} if there is a metric $d$ on $X$ such that $(X,d)$ is a complete, separable metric space with Borel $\sigma$-algebra $\mathcal{B}$. We say a m.p.s. $(X,\mathcal{B},\mu,(T_s)_{s\in S})$ is standard Borel if $(X,\mathcal{B})$ is standard Borel. 
\end{definition}

\begin{theorem}[See {\cite[Theorem 2.7.7]{Ec}} or {\cite[Theorem 2.9]{FJB}}]
\label{ExtendingActionsToGroupsNonCont} 
For any standard Borel m.p.s. $(X,\mathcal{B},\mu,(T_s)_{s\in S})$ there is a m.p.s. $(Y,\mathcal{C},\nu,(S_g)_{g\in G})$ and a measure preserving map $\pi:Y\to X$ with $\pi\circ S_s=T_s\circ\pi$ for all $s\in S$.
\end{theorem}

\begin{theorem}
\label{IFCSemigroups}
Let $S$ be a countably infinite, amenable, cancellative semigroup with a F{\o}lner sequence $(F_N)$\footnote{Any amenable, (two-sided) cancellative semigroup $S$ has a F{\o}lner sequence: let $G$ be the group of right fractions of $S$. Then from any F{\o}lner sequence $(F_N)$ in $G$ and any elements $c_N\in\cap_{g\in F_N}g^{-1}S$ we obtain a F{\o}lner sequence $(F_Nc_N)$ in $S$.}. For every m.p.s. $(X,\mathcal{B},\mu,(T_s)_{s\in S})$ and every $B\in\mathcal{B}$ there exists a subset $A\subseteq S$ such that for all $k\in\mathbb{N}$ and $h_1,\dots,h_k\in S$ we have
\begin{equation}
\label{IFCSemigroupsEq}
d_F\left(h_1^{-1}A\cap\dots\cap h_k^{-1}A\right)
=
\mu\left(T_{h_1}^{-1}B\cap\dots\cap T_{h_k}^{-1}B\right).
\end{equation}
\end{theorem}

\begin{remark}
We only generalize \Cref{IFC} to semigroup actions, but the same argument can be used to generalize \Cref{ThmFinAvs2.0}.
\end{remark}

\begin{proof}
Let $G$ be the group of right quotients of $S$. Note that as $S$ generates $G$ and 
$(F_N)$ is a F{\o}lner sequence in $S$, $(F_N)$ is also a F{\o}lner sequence in $G$.

We first suppose that  $(X,\mathcal{B})$ is standard Borel, so by \Cref{ExtendingActionsToGroupsNonCont} there is a m.p.s. $(Y,\mathcal{C},\nu,(S_g)_{g\in G})$ and a measure preserving map $\pi:Y\to X$ such that $\pi\circ S_s=T_s\circ\pi$ for all $s\in S$. Letting $D=\pi^{-1}(B)$, we have 
\begin{equation*}
\nu\left(S_{h_1}^{-1}D\cap\dots\cap S_{h_k}^{-1}D\right)
=
\mu\left(T_{h_1}^{-1}B\cap\dots\cap T_{h_k}^{-1}B\right)
\end{equation*}
for all $h_1,\dots,h_k\in S$. We finish the proof be applying \Cref{IFC} to obtain a set $A\subseteq G$ such that for all $k\in\mathbb{N}$ and $h_1,\dots,h_k\in G$ we have
\begin{equation*}
d_F\left(h_1A\cap\dots\cap h_kA\right)
=
\nu(S_{h_1}D\cap\cdots\cap S_{h_k}D)
=
\mu\left(T_{h_1}B\cap\dots\cap T_{h_k}B\right).
\end{equation*}
We now tackle the general case, where $(X,\mathcal{B},\mu)$ is an arbitrary measure space. Our construction is similar to the ones in \cite[Chapter 5, Section 3]{Par}. 
We may assume that $S$ contains the identity $e\in G$ (if not, we work with $S\cup\{e\}$). 
Consider the compact metric space $M=\{0,1\}^S$ with the product topology, and let $\mathcal{B}_M$ be the Borel $\sigma$-algebra in $M$, which is generated by the sets $M_h=\{(x_s)\in M;x_h=1\}$, $h\in S$. 
The map $\Phi:X\to M;x\mapsto(\chi_B(T_sx))_{s\in S}$ is measurable, and if we consider the continuous, measurable maps $U_h:M\to M;(x_s)_{s\in S}\mapsto (x_{sh})_{s\in S}$, then we have $\Phi\circ T_h=U_h\circ\Phi$ for all $h\in S$. 
Thus, the pushforward probability measure $\nu=\Phi_*\mu$ in $(M,\mathcal{B}_M)$ is invariant by $U_h$ for all $h\in S$. 
Letting $C=\{(x_s)\in M;x_e=1\}$, for all $h\in S$ we have $\Phi^{-1}(U_s^{-1}C)=T_s^{-1}B$. So for all $h_1,\dots,h_k\in S$,
\begin{equation*}
\nu\left(U_{h_1}^{-1}C\cap\cdots\cap U_{h_k}^{-1}C\right)
=
\mu\left(T_{h_1}^{-1}B\cap\dots\cap T_{h_k}^{-1}B\right).
\end{equation*}
As $(M,\mathcal{B}_M)$ is standard Borel, we are done.
\end{proof}

\section{Characterizations of vdC sets in amenable groups}
\label{SecCharact}
In this section we prove \Cref{ThmvdC} below, our main characterization theorem for vdC sets in countable amenable groups. \Cref{ThmvdC} gives a characterization of $F$-vdC sets analogous to \Cref{defvcDZ} but for any F{\o}lner sequence. In particular, it implies that the notion of $F$-vdC set is independent of the F{\o}lner sequence, thus answering the question in \cite[Section 4.2]{BL} of whether $F$-vdC implies vdC. See \Cref{rewf9ds0ol} for yet another characterization of $F$-vdC sets.

\begin{theorem}\label{ThmvdC}
Let $G$ be a countably infinite amenable group with a F{\o}lner sequence $(F_N)$. For a set $H\subseteq G\setminus\{e\}$, the following are equivalent:
\begin{enumerate}
    \item \label{ThmvdC1}$H$ is an $F$-vdC set.
    
    \item \label{ThmvdC2} For all sequences $(z_g)_{g\in G}$ of complex numbers in the unit disk $\mathbb{D}$,
    \begin{equation*}
    \lim_{N\to\infty}\frac{1}{|F_N|}
    \sum_{g\in F_N}z_{hg}\overline{z_g}=0\text{ for all }h\in H\textup{ implies }
    \lim_{N\to\infty}\frac{1}{|F_N|}\sum_{g\in F_N}z_{g}=0.
    \end{equation*}

    \item \label{ThmvdC2'} For all sequences $(z_g)_{g\in G}$ of complex numbers in $\mathbb{S}^1$,
    \begin{equation*}\label{EqRrg1'}
    \lim_{N\to\infty}\frac{1}{|F_N|}
    \sum_{g\in F_N}z_{hg}\overline{z_g}=0\text{ for all }h\in H\textup{ implies }
    \lim_{N\to\infty}\frac{1}{|F_N|}\sum_{g\in F_N}z_{g}=0.
    \end{equation*}

    \item\label{ThmvdC4} $H$ is a vdC set in $G$: for all m.p.s. $(X,\mathcal{A},\mu,(T_g)_{g\in G})$ and $f\in L^\infty(\mu)$,
\begin{equation}\label{eriojgdfklsmcvxdf}
\int_{X}f(T_hx)\cdot\overline{f(x)}d\mu(x)=0\text{ for all }h\in H\text{ implies }\int_{X}fd\mu=0.
\end{equation}

    \item\label{ThmvdC3}\emph{(Finitistic characterization)} For every $\varepsilon>0$ there is $\delta>0$ and finite sets $A\subseteq G,H_0\subseteq H$ such that for any $K\in\mathbb{N}$ and sequences $(z_{a,k})_{a\in G}$ in $\mathbb{D}$, for $k=1,\dots,K$, such that 
    \begin{equation*}
    \left|
    \frac{1}{K|A|}\sum_{k=1}^K\sum_{a\in A}z_{ha,k}\overline{z_{a,k}}
    \right|<\delta\text{ for all }h\in H_0,
    \end{equation*}
    we have
    \begin{equation*}
    \left|
    \frac{1}{K|A|}\sum_{k=1}^K\sum_{a\in A}z_{a,k}
    \right|<\varepsilon.
    \end{equation*}
\end{enumerate}
\end{theorem}

\begin{remark}
It is an interesting question whether, if we change `$f\in L^\infty(\mu)$' by `$f\in L^2(\mu)$' in \Cref{DefGvdC}, the resulting definition of vdC set in $G$ is equivalent. A similar question is posed in 
\cite[Question 3.7]{FT}, where they call the sets defined by the $L^2$ definition `sets of operatorial recurrence'. 
\end{remark}

The relationship between equidistribution and Cesaro averages is explained by \Cref{WeylCrit} below, which was introduced by Weyl in \cite{We}. 
\begin{definition}
Let $G$ be a countable amenable group with a F{\o}lner sequence $(F_N)$. We say that a sequence $(z_g)_{g\in G}$, with $z_g\in\mathbb{S}^1$ for all $g\in G$, is \textit{$F$-u.d. in $\mathbb{S}^1$} if for every continuous function $f:\mathbb{S}^1\to\mathbb{C}$ we have
\begin{equation*}
\lim_{N\to\infty}\frac{1}{|F_N|}\sum_{g\in F_N}f(z_g)=\int_{\mathbb{S}^1}fdm,
\end{equation*}
where $m$ is the uniform probability measure in $\mathbb{S}^1$. 
\end{definition}
That is, a sequence $(x_g)_{g\in G}$ in $\mathbb{T}$ is $F$-u.d. mod $1$ (as in \Cref{DefFudmod1}) iff $(e^{2\pi ix_g})_{g\in G}$ is $F$-u.d. in $\mathbb{S}^1$.

\begin{prop}[Weyl's criterion for uniform distribution]\label{WeylCrit} Let $G$ be a countable amenable group with a F{\o}lner sequence $(F_N)$. A sequence $(z_g)_{g\in G}$ in $\mathbb{S}^1$ is $F$-u.d. in $\mathbb{S}^1$ iff for all $l\in\mathbb{Z}\setminus\{0\}$ (or equivalently, for all $l\in\mathbb{N}$) we have
\begin{equation*}
\lim_N\frac{1}{|F_N|}\sum_{g\in F_N}z_g^l=0.
\end{equation*}
\end{prop}

For a proof of \Cref{WeylCrit} see e.g. \cite[Theorem 2.1]{KN} (the proof works for any F{\o}lner sequence).

\begin{proof}[Proof of \Cref{ThmvdC}]We prove \ref{ThmvdC2}$\implies$\ref{ThmvdC2'}$\implies$\ref{ThmvdC1}$\implies$\ref{ThmvdC3}$\implies$\ref{ThmvdC4}$\implies$\ref{ThmvdC2} 
\begin{enumerate}
    \item[\ref{ThmvdC2}$\implies$\ref{ThmvdC2'}] Obvious.
    \item[\ref{ThmvdC2'}$\implies$\ref{ThmvdC1}]
    If $H$ is not $F$-vdC then there is a sequence of complex numbers $(z_g)_{g\in G}$ in $\mathbb{S}^1$ which is not $F$-u.d. in $\mathbb{S}^1$ but such that $(z_{hg}\overline{z_g})_g$ is $F$-u.d. in $\mathbb{S}^1$ for all $h\in H$. By Weyl's criterion, that means that for all $h\in H$ we have
    \begin{equation*}
\lim_N\frac{1}{|F_N|}\sum_{g\in F_N}z_{hg}^l\overline{z_g^l}=0\text{ for all }l\in\mathbb{Z}\setminus\{0\},
\end{equation*}
but there is some $l_0\in\mathbb{Z}\setminus\{0\}$ such that
\begin{equation*}
\limsup_N\left|\frac{1}{|F_N|}\sum_{g\in F_N}z_g^{l_0}\right|>0.
\end{equation*}
The sequence $(z_g^{l_0})_{g\in G}$ contradicts \ref{ThmvdC2'}, so we are done.
\end{enumerate}

\item[\ref{ThmvdC4}$\implies$\ref{ThmvdC2}]
Suppose a sequence $(z_g)_{g\in G}$ does not satisfy \ref{ThmvdC2}. Taking a F{\o}lner subsequence $(F_N')_{n\in\mathbb{N}}$ if necessary, we can assume that we have
\begin{gather*}
    \lim_{N\to\infty}\frac{1}{|F_N'|}\sum_{g\in F_N'}z_g=\lambda\neq0\\
    \lim_{N\to\infty}
    \frac{1}{|F_N'|}
    \sum_{g\in F_N'}z_{hg}\overline{z_g}=0\text{ for all }h\in H.
\end{gather*}
This contradicts \ref{ThmvdC4} by \Cref{ThmFinAvs2.0} applied with $D=\mathbb{D}$ to the Cesaro averages of $(z_g)_{g\in G}$ and $(z_{hg}\overline{z_g})_{g\in G}$.

\item[\ref{ThmvdC3}$\implies$\ref{ThmvdC4}]
The contrapositive of this implication follows from \Cref{ThmFinAvs2.0} applied to Cesaro averages of the functions $(z_g)_{g\in G}$ and $(z_{hg}\overline{z_g})_{g\in G}$, with $D=\mathbb{D}$. Indeed, if $H\subseteq G$ does not satisfy \ref{ThmvdC4}, then there is a m.p.s. $(X,\mathcal{A},\mu,(T_g)_{g\in G})$ and $f\in L^\infty(\mu)$ such that, for some $\lambda\neq0$
\begin{equation*}
\int_{X}f(T_hx)\cdot\overline{f(x)}d\mu(x)=0\text{ for all }h\in H\text{ but }\int_{X}fd\mu=\lambda.
\end{equation*} 
So by \Cref{ThmFinAvs2.0}\Cref{ThmFinAvsGen2}$\implies$\Cref{ThmFinAvsGen3}, for any finite sets $A\subseteq G,H_0\subseteq H$ there is  $K\in\mathbb{N}$ and sequences $(z_{a,k})_{a\in G}$ in $\mathbb{D}$, for $k=1,\dots,K$, such that 
    \begin{equation*}
    \left|
    \frac{1}{K|A|}\sum_{k=1}^K\sum_{a\in A}z_{ha,k}\overline{z_{a,k}}
    \right|<\delta\text{ for all }h\in H_0,
    \end{equation*}
    but
    \begin{equation*}
    \left|
    \frac{1}{K|A|}\sum_{k=1}^K\sum_{a\in A}z_{a,k}-\lambda
    \right|<\delta,
    \end{equation*}
    contradicting \ref{ThmvdC4}.

\item[\ref{ThmvdC1}$\implies$\ref{ThmvdC3}]
Suppose that \ref{ThmvdC3} does not hold for some $\varepsilon>0$. So for any finite sets $A\subseteq G$ and $H_0\subseteq H$ and for any $\delta>0$ there exist $K\in\mathbb{N}$ and sequences $(z_{a,k})_{a\in G}$ in $\mathbb{D}$, for $k=1,\dots,K$, such that 
    \begin{equation*}
    \left|
    \frac{1}{K|A|}\sum_{k=1}^K\sum_{a\in A}z_{ha,k}\overline{z_{a,k}}
    \right|<\delta\text{ for all }h\in H_0,
    \end{equation*}
    but (we can assume that the following average is a positive real number)
    \begin{equation*}
    \frac{1}{K|A|}\sum_{k=1}^K\sum_{a\in A}z_{a,k}>\varepsilon.
    \end{equation*}

In fact, we can assume the even stronger 
\begin{equation*}
\left|\frac{1}{K|A|}\sum_{k=1}^K\sum_{a\in A}z_{a,k}-\varepsilon\right|
<\delta;
\end{equation*}
This can be achieved by adding some sequences $(z_{a,k})_{a\in G}$ ($k=K+1,\dots,K+K'$) with $z_{a,k}=0$ for all $a\in G$.\footnote{It may also be necessary to change $K$ by a multiple of $K$, by repeating each sequence $(z_{a,k})_a$ ($k=1,\dots,K$) several times} We will now prove the following:

\begin{enumerate}[label=$\neg$\arabic*']
\setcounter{enumi}{4}
\item \label{3''}
There is $\varepsilon>0$ such that, for any finite sets $A\subseteq G$ and $H_0\subseteq H$ and for any $L\in\mathbb{N}$ and $\delta>0$ there exist $J\in\mathbb{N}$ and sequences $(w_{a,j})_{a\in G}$ in $\mathbb{S}^1$, for $j=1,\dots,J$, such that 
    \begin{equation}\label{45trgdfvcythfgcv45trgdf}
    \left|
    \frac{1}{J|A|}\sum_{j=1}^J\sum_{a\in A}w_{ha,j}^l\overline{w_{a,j}^l}
    \right|<\delta\text{ for all }h\in H_0,l=1,\dots,L.
    \end{equation}
    but 
    \begin{equation}\label{lkjefwsdclksdfl.3}
    \left|\frac{\varepsilon}{2}-
    \frac{1}{J|A|}\sum_{j=1}^J\sum_{a\in A}w_{a,j}
    \right|<\delta.
    \end{equation}
\end{enumerate}
First we note that \ref{3''} implies $\neg$\ref{ThmvdC1}, due to \Cref{ThmFinAvs2.0} applied to the Cesaro averages of $(z_g)_g$ and $(z_{hg}^l\overline{z_g^l})_g$, for $h\in H$ and $l\in\mathbb{Z}$, and Weyl's criterion. Let us now prove \ref{3''} using a probabilistic trick from \cite[Section 6]{Ru}

Let $(z_{a,k})_{a\in G}$, $k=1,\dots,K$, be as above, and let $\delta_1>0$. We will consider a family of independent random variables $(\xi_{a,k})_{a\in G,k=1,\dots,K}$ supported in $\mathbb{S}^1$, with $\xi_{a,k}$ having density function $d_{a,k}:\mathbb{S}^1\to[0,1];\;z\mapsto 1+\text{Re}(z\overline{z_{a,k}})$. Then,
\begin{equation}\label{45tereyyry65612}
\mathbb{E}(\xi_{a,k})=\frac{z_{a,k}}{2}\text{ and }\mathbb{E}(\xi_{a,k}^n)=0\text{ if }n=\pm2,\pm3,\dots.
\end{equation}
\Cref{45tereyyry65612} can be proved when $z_{a,k}=1$ integrating, as we have $\int_0^1e^{2\pi ix}(1+\cos(2\pi x))=\frac{1}{2}$ and for $n=\pm2,\pm3,\dots$, $\int_0^1e^{2\pi inx}(1+\cos(2\pi x))=0$. For other values of $z_{a,k}$ one can change variables to $w=z\overline{z_{a,k}}$. So we have 
\begin{equation}\label{45tereyyry65612'}
\mathbb{E}(\xi_{ha,k}\overline{\xi_{a,k}})=\frac{z_{ha,k}\overline{z_{a,k}}}{4}\text{ and }\mathbb{E}(\xi_{ha,k}^l\overline{\xi_{a,k}^l})=0\text{ if }l=\pm2,\pm3,\dots.
\end{equation}
Now, for each $m\in\mathbb{N}$ we define a sequence $(z_{a,k,m})_{a\in G}$ by choosing, independently for all $k=1,\dots,K,a\in G$ and $m\in\mathbb{N}$, a complex number $z_{a,k,m}\in\mathbb{S}^1$ according to the distribution of $\xi_{a,k}$. Then the strong law of large numbers implies that with probability $1$ we will have, for all $k=1,\dots,K,l\in\mathbb{Z},h\in H$ and $a\in G$,
\begin{equation}\label{CLTapp}
\lim_{M\to\infty}
\left|
\mathbb{E}\left(\xi_{ha,k}^l\overline{\xi_{a,k}^l}\right)-\frac{1}{M}\sum_{m=1}^Mz_{ha,k,m}^l\overline{z_{a,k,m}^l}
\right|=0.
\end{equation}
So we can fix a family $(z_{a,k,m})_{a\in G;k=1,\dots,K;m\in\mathbb{N}}$ such that \Cref{CLTapp} holds for all $k,l,h,a$ as above. Then, taking averages over all $a\in A$ and $k=1,\dots,K$, we obtain that for all $h\in H_0$ and $l\in\mathbb{Z}$
\begin{equation*}
\lim_{M\to\infty}
\left|
\frac{1}{K|A|}\sum_{k=1}^K\sum_{a\in A}\mathbb{E}\left(\xi_{ha,k}^l\overline{\xi_{a,k}^l}\right)-\frac{1}{MK|A|}\sum_{k=1}^K\sum_{a\in A}\sum_{m=1}^Mz_{ha,k,m}^l\overline{z_{a,k,m}^l}
\right|=0.
\end{equation*}
Notice that, due to \Cref{45tereyyry65612'}, for $l>1$ the expression in the LHS is $0$, and for $l=1$ it is $\frac{1}{K|A|}\sum_{k=1}^K\sum_{a\in A}\frac{z_{ha,k}\overline{z_{a,k}}}{4}$, which has norm $<\frac{\delta}{4}$. 
So taking some big enough value $M$, taking the sequences $(w_{a,j})_{a\in G}$ to be the sequences $(z_{a,k,m})_{a\in G}$, for $k=1,\dots,K$ and $m=1,\dots,M$, \Cref{45trgdfvcythfgcv45trgdf} will be satisfied with $J=KM$ for all $h\in H_0$ and $l=1,\dots,L$. We can check similarly that, for a big enough value of $M$, \Cref{lkjefwsdclksdfl.3} will be satisfied by the sequences $(z_{a,k,m})_{a\in G}$, for $k=1,\dots,K$ and $m=1,\dots,M$, so we are done.
\end{proof}

We will need a finitistic criterion for the notion of vdC sets in order to prove a property of vdC sets (\Cref{rtdfgt5ergdfvc}). Letting $\mathbb{D}:=\{z\in\mathbb{C};|z|\leq1\}$, we have
\begin{prop}\label{trefds5rtefdsc}
Let $G$ be a countably infinite group, let $H\subseteq G$. Then $H$ is a vdC set in $G$ if and only if for any $\varepsilon>0$  there exists $\delta>0$ and a finite subset $H_0$ of $H$ such that, for any m.p.s. $(X,\mathcal{B},\mu,(T_g)_{g\in G})$ and any measurable $f:X\to\mathbb{D}$ we have
\begin{equation*}\left|\int_Xf(T_hx)\overline{f(x)}d\mu(x)\right|<\delta\;\forall h\in H_0\text{ implies }\left|\int_Xfd\mu\right|<\varepsilon.
\end{equation*}
\end{prop}

\begin{proof}
A natural proof of this fact uses Loeb measures. In order to avoid using non-standard analysis, we will adapt the proof of \cite[Lemma 6.4]{Fo}.

Let $H_n$ be an increasing sequence of finite sets with $H=\cup_nH_n$. Suppose we have $\varepsilon>0$ and a sequence of m.p.s.'s $(X_n,\mathcal{B}_n,\mu_n,(T_{n,g})_{g\in G})$ and measurable functions $f_n:X\to\mathbb{D}$ such that 
\begin{equation*}
\left|\int_{X_n}f_n(T_{n,h}x)\overline{f_n(x)}d\mu_n(x)\right|<\frac{1}{n}\;\forall h\in H_n,\text{ but }\left|\int_{X_n}f_n(x)d\mu_n(x)\right|>\varepsilon.
\end{equation*}
We will prove that $H$ is not a vdC set by constructing a m.p.s. $(Y,\mathcal{C},\nu,(S_g)_{g\in G})$ and some measurable $f_\infty:Y\to\mathbb{D}$ such that 
\begin{equation}\label{GelfandTrick}
\left|\int_{Y}f_\infty(S_{h}y)\overline{f_\infty(y)}d\nu(y)\right|=0\;\forall h\in H,\text{ but }\left|\int_{Y}f_\infty(y)d\nu(y)\right|\geq\varepsilon.
\end{equation}
To do it, first consider the (infinite) measure space $(X,\mathcal{B},\mu,(T_g)_{g\in G})$ with $X$ being the disjoint union $\sqcup_{n\in\mathbb{N}}X_n$, $\mathcal{B}:=\left\{\sqcup_nB_n;B_n\in\mathcal{B}_n\;\forall n\right\}$, $\mu(\sqcup_nB_n):=\sum_n\mu_n(B_n)$ and $T_gx=T_{n,g}x$ if $x\in X_n$. And let $f:X\to\mathbb{D}$ be given by $f|_{X_n}=f_n$.

Let $K$ be the smallest sub-$C^*$-algebra of $L^\infty(X)$ containing $1$ and $T_gf$ for all $g\in G$. Then by the Gelfand representation theorem there is a $C^*$-algebra isomorphism $\Phi:K\to C(Y)$, where $Y$ is the spectrum of $K$, a compact metrizable space whose elements are non-zero $*$-homomorphisms $y:K\to\mathbb{C}$. The isometry $\Phi$ is given by $\Phi(\varphi)(y)=y(\varphi)$.

We let $\mathcal{C}$ be the Borel $\sigma$-algebra of $Y$. Now consider a Banach limit $L:l^\infty(\mathbb{N})\to\mathbb{C};(a_n)\mapsto L-\lim_na_n$ and define a norm $1$ positive  functional $F:K\to\mathbb{C}$ by 
\begin{equation*}
F(\varphi)=L-\lim_n\int_{X_n}\varphi(x)d\mu_n(x).
\end{equation*}
This induces a probability measure $\nu$ on $Y$ by $\int_Y\Phi(\varphi)d\nu=F(\varphi)$. Also, the action of $G$ on $K$ induces an action $(S_g)_{g\in G}$ of $G$ on $Y$ by homeomorphisms by $(S_gy)(\varphi)=y(\varphi\circ T_g)$. $S_g$ is $\nu$-preserving for all $g\in G$, as for any $\varphi\in K$ we have
\begin{multline*}
\int_{Y}\Phi(\varphi)(S_gy)d\nu(y)=
\int_{Y}\Phi(\varphi\circ T_g)(y)d\nu(y)\\=F(\varphi\circ T_g)=F(\varphi)=\int_{Y}\Phi(\varphi)(y)d\nu(y).
\end{multline*}
We now let $f_\infty:=\Phi(f)\in C(Y)$ and we check \Cref{GelfandTrick}: for all $h\in H$,
\begin{multline*}
\int_Y\Phi(f)(S_hy)\overline{\Phi(f)(y)}d\nu(y)
=
\int_Y\Phi(f\circ T_h)(y)\overline{\Phi(f)(y)}d\nu(y)
\\=
F((f\circ T_h)\cdot\overline{f})
=
L-\lim_n\int_{X_n}f(T_{n,h}x)\overline{f_n(x)}d\mu_n(x)=0.
\end{multline*}
\begin{multline*}
\left|\int_{Y}\Phi(f)(y)d\nu(y)\right|
=
\left|F(f)\right|
=
\left|
L-\lim_n\int_{X_n}
f(x)d\mu_n(x)
\right|
\\=
L-\lim_n
\left|
\int_{X_n}
f(x)d\mu_n(x)
\right|
\geq\varepsilon.
\end{multline*}
\end{proof}

\section{Spectral Characterization of vdC sets in abelian groups}\label{SecSpect}

In this section we prove a spectral criterion for vdC sets in countable abelian groups (\Cref{SpeCrit}), which is a direct generalization of the spectral criterion obtained in \cite[Theorem 1]{Ru}; we state it in a fashion similar to \cite[Theorem 8]{BL}. 
\Cref{SpeCrit} implies that the notion of \emph{vdC set in $\mathbb{Z}^d$} defined in \cite{BL} is equivalent to our notion of vdC set, even if it is not defined in terms of a F{\o}lner sequence (their Definition 2 uses instead the F{\o}lner net of rectangles $R_{M,N}=[0,M]\times[0,N]$, $M,N\in\mathbb{N}$, ordered by $R_{M,N}\geq R_{M',N'}$ if $M\geq M'$ and $N\geq N'$). 
Also see \cite[Theorem 4.3]{FT} for a different proof of \Cref{SpeCrit}, shorter than the one included here.

\begin{customthm}{\ref{SpeCrit}}
Let $G$ be a countable abelian group. A set $H\subseteq G\setminus\{0\}$ is a vdC set in $G$ iff any Borel probability measure $\mu$ in $\widehat{G}$ with $\widehat{\mu}(h)=0\;\forall h\in H$ satisfies $\mu\left(\left\{1_{\widehat{G}}\right\}\right)=0$.
\end{customthm}
As in \cite[Theorem 1.8]{BL} it follows from \Cref{SpeCrit} that, if $H$ is a vdC set in $G$ and a Borel probability measure $\mu$ in $\widehat{G}$ satisfies $\widehat{\mu}(h)=0$ for all $h\in H$, then $\mu(\{\gamma\})=0$ for all $\gamma\in\widehat{G}$.

\begin{proof}
\begin{enumerate}    \item[$\implies$] 
    Suppose there is a probability measure $\mu$ in $\widehat{G}$ such that $\widehat{\mu}(h)=0$ for all $h\in H$ but $\mu\left(\left\{1_{\widehat{G}}\right\}\right)=\lambda>0$. 

    Consider a sequence of finitely supported measures $(\mu_N)_{n\in\mathbb{N}}$ which converge weakly to $\mu$; we can suppose $\mu_N\left(\left\{1_{\widehat{G}}\right\}\right)\to\lambda$. Specifically, $\mu_N$ will be an average of Dirac measures
\begin{equation*}
\mu_N:=\frac{1}{u_N}\sum_{i=1}^{u_N}\delta_{x_{N,i}},
\end{equation*}
for some natural numbers $(u_N)_{n\in\mathbb{N}}\to\infty$ and $x_{N,1},\dots,x_{N,u_N}\in \widehat{G}$, so that $x_{N,i}=1_{\widehat{G}}\in \widehat{G}$ iff $i\leq u_N\lambda$. The fact that $\mu_N\to\mu$ weakly implies that for all $h\in H$ (seeing $h$ as a map $h:\widehat{G}\to\mathbb{S}^1$) we have 
\begin{equation}\label{ejiudgrjotlrxujiltdfi}
\lim_N\frac{1}{u_N}\sum_{i=1}^{u_N}x_{N,i}(h)=\lim_N\int_{\widehat{G}}hd\mu_N=\int_{\widehat{G}}hd\mu=\widehat{\mu}(h)=0.
\end{equation}
We will prove that, letting $\varepsilon<\lambda$, \Cref{ThmvdC3} of \Cref{ThmvdC} is not satisfied. So let $A\subseteq G$ and $H_0\subseteq H$ be finite and let $(F_N)$ be a F{\o}lner sequence  in $G$. For each $N$ we consider the sequences $(x_{N,i}(ag))_{a\in A}$ for all $i=1,\dots,u_N$ and $g\in F_N$. It will be enough to prove that
\begin{equation}\label{Eq68orsth}
\lim_{N\to\infty}\frac{1}{u_N|F_N|\cdot|A|}\sum_{a\in A,g\in F_N,i=1,\dots,u_N}x_{N,i}(ag)=\lambda,
\end{equation}
and for all $h\in H_0$,
\begin{equation}\label{ytrhfg54terdfg}
\lim_{N\to\infty}\frac{1}{u_N|F_N|\cdot|A|}\sum_{a\in A,g\in F_N,i=1,\dots,u_N}x_{N,i}(hag)\overline{x_{N,i}(ag)}=0.
\end{equation}
\Cref{ytrhfg54terdfg} is a direct consequence of \Cref{ejiudgrjotlrxujiltdfi} and the fact that $x_{N,i}(hag)\overline{x_{N,i}(ag)}=x_{N,i}(h)$. To prove \Cref{Eq68orsth} first note that, as $(F_N)$ is a F{\o}lner sequence, \Cref{Eq68orsth} is equivalent to
\begin{equation}\label{657ythfgdertryhgfj}
\lim_{N\to\infty}\frac{1}{u_N|F_N|}\sum_{g\in F_N,i=1,\dots,u_N}x_{N,i}(g)=\lambda.
\end{equation}
Now, let $\delta>0$ and consider a neighborhood $U$ of $1_{\widehat{G}}$ with $\mu\left(\overline{U}\right)<\lambda+\delta$. 
After a reordering, we can assume that for big enough $N$ the points $x_{N,i}$ are in $\widehat{G}\setminus U$ for all $i\geq u_N(\lambda+2\delta)$.
\begin{claim}\label{5trgdfvc}
There exists $M\in\mathbb{N}$ such that, for all $x\in \widehat{G}\setminus U$ and for all $N\geq M$, $\left|\frac{1}{|F_N|}\sum_{g\in F_N}x(g)\right|<\delta$.
\end{claim}
\Cref{5trgdfvc} implies \Cref{657ythfgdertryhgfj}, because $\delta$ is arbitrary and by \Cref{5trgdfvc} we have that for all $i\leq u_N\lambda$, the average $\frac{1}{|F_N|}\sum_{g\in F_N}x_{N,i}(g)$ is exactly $1$ (as $x_{N,i}=1$), and for all $i\geq u_N(\lambda+2\delta)$ and big enough $N$, $\frac{1}{|F_N|}\sum_{g\in F_N}x_{N,i}(g)$ has norm at most $\delta$.

To prove \Cref{5trgdfvc} let $\varepsilon>0$ and $g_1,\dots,g_k\in G$ be such that for all $x\in\widehat{G}\setminus U$ we have $|x(g_i)-1|>\varepsilon$ for some $i\in\{1,\dots,k\}$. Now note that for all $N\in\mathbb{N}$, $x\in\widehat{G}\setminus U$ and $i=1,\dots,k$, 
\begin{align*}
|x(g_i)-1|\cdot\left|\frac{1}{|F_N|}\sum_{g\in F_N}x(g)\right|
&=
\frac{1}{|F_N|}\left|\sum_{g\in F_N}x(g_ig)-\sum_{g\in F_N}x(g)\right|\\
&\leq
\frac{1}{|F_N|}\sum_{g\in F_N\Delta g_iF_N}\left|x(g)\right|.
\end{align*}
As $\lim_N\frac{|F_N\Delta g_iF_N|}{|F_N|}=0$ for all $i$, there is some $M$ such that for all $N\geq M$ and for all $i$, the right hand side is $<\delta\varepsilon$  for all $x\in\widehat{G}\setminus U$. Thus, $M$ satisfies \Cref{5trgdfvc}.

\item[$\impliedby$]
The proof of \cite[Theorem 1.8, S2$\implies $S1]{BL} can be adapted to any countable abelian group; instead of \cite[Lemma 1.9]{BL} one needs to prove a statement of the form
\begin{lema}
Let $(u_g)_{g\in G}$ be a sequence of complex numbers in $\mathbb{D}$ and let $(F_N)$ be a F{\o}lner sequence such that, for all $h\in G$, the value
\begin{equation*}
\gamma(h)=\lim_N\frac{1}{|F_N|}\sum_{g\in F_N}u_{hg}\overline{u_g}
\end{equation*} is defined. Then there is a positive measure $\sigma$ on $\widehat{G}$ such that $\widehat{\sigma}(h)=\gamma(h)$ for all $h$, and \begin{equation*}
\limsup_N\frac{1}{|F_N|}\left|\sum_{g\in F_N}u_g\right|\leq\sqrt{\sigma(\{0\})}.
\end{equation*}
\end{lema}

And the lemma can also be proved similarly to \cite[Lemma 1.9]{BL}, changing the functions $g_N,h_N$ from \cite{BL} by
\begin{equation*}
g_N(x)=\frac{1}{|F_N|}\left|\sum_{g\in F_N}z_gx(g)\right|^2\text{ and }
h_N(x)=\frac{1}{|F_N|}\left|\sum_{g\in F_N}g(x)\right|^2,
\end{equation*}
and one also needs to check that \cite[Theorem 2]{CKM} works for any countable abelian group:
\begin{lema}[Cf. {\cite[Theorem 2]{CKM}}]\label{4erfdvcx45rtegdfv}
Let $G$ be a countable abelian group, let $\widehat{G}$ be its dual. Let $(\mu_n)_{n\in\mathbb{N}},(\nu_n)_{n\in\mathbb{N}},\mu,\nu$ be Borel probability measures in $\widehat{G}$ such that $\mu_n\to\mu$ weakly, $\nu_n\to\nu$ weakly. Then
\begin{equation*}
\rho(\mu,\nu)\geq\limsup_n\rho(\mu_n,\nu_n),
\end{equation*}
where $\rho(\mu,\nu)$, the \textit{affinity} between $\mu$ and $\nu$, is given by, for any measure $m$ such that $\mu,\nu$ are absolutely continuous with respect to $m$,
\begin{equation*}
\rho(\mu,\nu)=\int_{\widehat{G}}\left(\frac{d\mu}{dm}\right)^{\frac{1}{2}}\left(\frac{d\nu}{dm}\right)^{\frac{1}{2}}dm.
\end{equation*}
\end{lema}
The same proof of \cite[Theorem 2]{CKM} is valid; the proof uses the existence of Radon-Nikodym derivatives and a countable partition of unity $f_j:T\to\mathbb{R}$, for $j\in\mathbb{Z}$. These partitions of unity always exist for outer regular Radon measures (see \cite[Theorem 3.14]{RuRC}), so they exist for any Borel probability measure in $\widehat{G}$.\qedhere
\end{enumerate}
\end{proof}

\section{Properties of vdC sets}\label{SecProps}
In \cite{BL} and \cite{Ru}, several properties of the family of vdC subsets of $\mathbb{Z}^d$ were proved. In this section we check that many of these properties hold for vdC sets in any countable group. Some of the statements about vdC sets follow from statements about sets of recurrence and the fact that any vdC set is a set of recurrence. 
Other properties can be proved in the same way as their analogs for sets of recurrence, even if they are not directly implied by them.

\begin{remark}
All the properties below are also satisfied for sets of operatorial recurrence, as defined in \cite{FT} (see \cite[Theorem 4.4]{FT}), with the proofs of the properties being essentially the same as for vdC sets.
\end{remark}

\begin{prop}[Partition regularity of vdC sets, cf. {\cite[Cor. 1]{Ru}}]\label{RamseyGvdC}
Let $G$ be a countably infinite group and let $H,H_1,H_2\subseteq G\setminus\{e\}$ satisfy $H=H_1\cup H_2$. If $H$ is a vdC set in $G$, then either $H_1$ or $H_2$ are vdC sets in $G$.
\end{prop}

\begin{proof}
Suppose that $H_1,H_2$ are not vdC sets in $G$. Then for $i=1,2$ there are measure preserving systems $(X_i,\mathcal{B}_i,\mu_i,(T_i^g)_{g\in G})$, and functions $f_i\in L^\infty(\mu_i)$ such that
\begin{equation*}
\int_{X_i}f_i(T_i^hx)\cdot\overline{f_i(x)}d\mu_i(x)=0\text{ for all }h\in H_i,
\end{equation*}
but 
\begin{equation*}
\int_{X_i}f_i(x)d\mu_i(x)\neq0.
\end{equation*}
But then, considering the function $f:X_1\times X_2\to\mathbb{C};f(x_1,x_2)=f_1(x_1)\cdot f_2(x_2)$, we have by Fubini's theorem
\begin{equation*}
\int_{X_1\times X_2}f_1(T_1^hx_1)f_2(T_2^hx_2)\overline{f_1(x_1)}\overline{f_2(x_2)}d(\mu_1\times\mu_2)(x_1,x_2)=0\text{ for all }h\in H_1\cup H_2,
\end{equation*}
but 
\begin{equation*}
\int_{X_1\times X_2}f_1(x_1)f_2(x_2)d(\mu_1\times\mu_2)(x_1,x_2)\neq0,
\end{equation*}
so $H_1\cup H_2$ is not a vdC set in $G$.
\end{proof}

We will need the following in order to prove \Cref{rtdfgt5ergdfvc}.
\begin{prop}\label{finGvdC}
Let $G$ be a countably infinite group, and let $H\subseteq G\setminus\{e\}$ be finite. Then $H$ is not a set of recurrence in $G$, so it is not vdC.
\end{prop}

\begin{proof}
Consider the $\left(\frac{1}{2},\frac{1}{2}\right)$-Bernoulli scheme in $\{0,1\}^G$ with the action $(T_g)_{g\in G}$ of $G$ given by $T_g((x_a)_{a\in G})=(x_{ag})_{a\in G}$. Let $A=\{(x_a);x_e=0\}\subseteq\{0,1\}^G$ and let $B=A\setminus\cup_{h\in H}T_{h^{-1}}A$. As $H$ is finite, $B$ has positive measure, but clearly $\mu(B\cap T_{h^{-1}}B)=0$ for all $h\in H$, so we are done.
\end{proof}

The following generalizes {\cite[Cor. 3]{Ru}}, and has a similar proof.
\begin{prop}\label{rtdfgt5ergdfvc}
Let $G$ be a countably infinite group and let $H\subseteq G\setminus\{e\}$ be a vdC set in $G$. Then we can find infinitely many disjoint vdC subsets of $H$.
\end{prop}

\begin{proof}[Proof of \Cref{rtdfgt5ergdfvc}]
It will be enough to prove that there are two disjoint vdC subsets $H',H''$ of $H$.

We will define a disjoint sequence $H_1,H_2,\dots$ of finite subsets of $H$ by recursion. Suppose $H_1,\dots,H_{n-1}$ are given; notice that $\bigcup_{i=1}^{n-1}H_i$ is finite, so it is not a vdC set. So by \Cref{RamseyGvdC}, $H\setminus\bigcup_{i=1}^{n-1}H_i$ is a vdC set. Then by \Cref{trefds5rtefdsc} we can then define $H_n$ to be a finite subset of $H\setminus\bigcup_{i=1}^{n-1}H_i$ such that for some constant $\delta_n>0$ and for any m.p.s. $(X,\mathcal{B},\mu,(T_g)_{g\in G})$, $f\in L^\infty(X,\mu)$ we have
\begin{equation*}\left|\int_Xf(T_hx)\overline{f(x)}d\mu(x)\right|<\delta_n\;\forall h\in H_n\text{ implies }\left|\int_Xfd\mu\right|<\frac{1}{n}.
\end{equation*}
Now let $H'=\bigcup_{n\in\mathbb{N}}H_{2n-1}$ and $H''=\bigcup_{n\in\mathbb{N}}H_{2n}$. It is then clear that both $H'$ and $H''$ satisfy the definition of vdC set in $G$, so we are done.
\end{proof}

The following generalizes \cite[Corollary 1.15.2]{BL} to countable amenable groups.
\begin{theorem}\label{y6thgft5rgfd}
Let $S$ be a subgroup of a countable amenable group $G$, let $H\subseteq S\setminus\{e\}$. Then $H$ is a vdC set in $S$ iff it is a vdC set in $G$.
\end{theorem}

\begin{proof}
If $H$ is not a vdC set in $G$, then it is not a vdC set in $S$, as any measure preserving action $(T_g)_{g\in G}$ on a measure space restricts to a measure preserving action $(T_g)_{g\in S}$ on the same measure space.

The fact that if $H$ is not a vdC set in $S$ then it is not a vdC set in $G$ can be deduced from \Cref{ThmvdC3} of \Cref{ThmvdC}: indeed, let $\varepsilon$ be as in \Cref{ThmvdC3} (applied to the group $S$) and consider any $A\subseteq G$ and $H_0\subseteq H$ finite and any $\delta>0$. 

We can express $A=A_1\cup\dots\cup A_m$, where the $A_i$ are pairwise disjoint and of the form $A_i=A\cap(Sg_i)$, for some $g_i\in G$. Thus, $Ag_i^{-1}\subseteq S$ for all $i$.

Now, for each $i$ we know that there exist $K_i\in\mathbb{N}$ and $G$-sequences $(z_{i,a,k})_{a\in S}$ in $\mathbb{D}$, for $k=1,\dots,K_i$, such that 
    \begin{equation}\label{gfhygfff3lemi}
    \left|
    \frac{1}{K_i|A_i|}\sum_{k=1}^{K_i}\sum_{a\in A_ig_i^{-1}}z_{i,ha,k}\overline{z_{i,a,k}}
    \right|<\delta\text{ for all }h\in H_0,
    \end{equation}
    but 
    \begin{equation}\label{gfhygfff4lemi}
    \left|
    \frac{1}{K_i|A_i|}\sum_{k=1}^{K_i}\sum_{a\in A_ig_i^{-1}}z_{i,a,k}
    \right|>\varepsilon.
    \end{equation}
Note that we can assume $K_1,K_2,\dots,K_m$ are all equal to some number $K$ (e.g. taking $K$ to be the least common multiple of all of them) and that $
    \frac{1}{K|A_i|}\sum_{k=1}^{K}\sum_{a\in A_ig_i^{-1}}z_{i,a,k}$ is a positive real number for all $i=1,\dots,m$ (multiplying the sequences $(z_{i,a,k})$ by some complex number of norm $1$ if needed).

    Finally, define for each $k=1,\dots,K$ a sequence $(z_{a,k})_{a\in G}$ by $z_{a,k}=z_{i,ag_i^{-1},k}$ for $a\in Sg_i$ and by $z_g=0$ elsewhere. 
    This sequence will satisfy \Cref{ThmvdC3} of \Cref{ThmvdC} (by taking averages of \Cref{gfhygfff3lemi,gfhygfff4lemi} for $k=1,\dots,K$), so we are done.
\end{proof}

The following is proved in \cite[Cor. 1.15.1]{BL} for vdC sets in $\mathbb{Z}^d$ using the spectral criterion; using \Cref{DefGvdC} instead we prove it for any countable group.
\begin{prop}\label{4refdstrgfd}
Let $\pi:G\to S$ be a group homomorphism, let $H$ be a vdC set in $G$. If $e_S\not\in\pi(H)$, then $\pi(H)$ is a vdC set in $S$.
\end{prop}

\begin{proof}
The contra-positive of the claim follows easily from the fact that any measure preserving action $(T_s)_{s\in S}$ on a probability space $(X,\mathcal{A},\mu)$ induces a measure preserving action $(S_g)_{g\in G}$ on $(X,\mathcal{A},\mu)$ by $S_g=T_{\pi(g)}$.
\end{proof}

\begin{remark}
In \Cref{4refdstrgfd}, $\pi(H)$ may be a vdC set even if $H$ is not. Indeed, the set $\{(n,1);n\in\mathbb{N}\}$ is not a set of recurrence in $\mathbb{Z}^2$ (this follows from \Cref{cor58} below), but its projection to the first coordinate is a vdC set.
\end{remark}

The following generalizes \cite[Corollary 1.16]{BL}.
\begin{cor}
\label{cor58}
Let $G$ be a countable group and let $S$ be a finite index subgroup of $G$. Then $G\setminus S$ is not a set of recurrence in $G$, so it is not a vdC set in $G$.
\end{cor}

\begin{proof}
Consider the action of $G$ on the finite set $G/S$ of left-cosets of $S$, where  we give $G/S$ the uniform probability measure $\mu$. The set $\{S\}\subseteq G/S$ has positive measure, but for all $g\in G\setminus S$ we have $g\{S\}\cap\{S\}=\varnothing$. 
\end{proof}

We finally prove that difference sets are nice vdC sets (see e.g. \cite[Lemma 5.2.8]{Fa} for the case $G=\mathbb{Z}$). The proof of \Cref{y6rtfgd5trefd} is just the proof that any set of differences is a set of recurrence\footnote{A slight modification of this proof shows that $AA^{-1}$ is a set of nice recurrence (see \Cref{DefNiceRec})}, which is already found in \cite[Page 74]{Fu} for the case $G=\mathbb{Z}$.

\begin{prop}\label{y6rtfgd5trefd}
Let $G$ be a countable group and let $A\subseteq G$ be infinite. Then the difference set $AA^{-1}=\{ba^{-1};a,b\in A,a\neq b\}$ is a nice vdC set in $G$.
\end{prop}

\begin{proof}
Suppose that for some m.p.s. $(X,\mathcal{A},\mu,(T_g)_{g\in G})$ and some function $f\in L^\infty(\mu)$ we have
\begin{equation}\label{trgdftre4}
\left|\int_Xfd\mu\right|^2
>
\limsup_{h\in AA^{-1}}\left|\int_Xf(T_hx)\overline{f(x)}d\mu(x)\right|.
\end{equation}
So for some $\lambda<\left|\int_Xfd\mu\right|^2$ and some finite subset $B\subseteq AA^{-1}$ we have
\begin{equation*}
\left|
\int_{X}f( T_ax)\cdot\overline{f(T_bx)}d\mu(x)\right|<\lambda\text{ for all }a,b\in A\text{ with }ab^{-1}\not\in B.
\end{equation*}
Then for any $N\in\mathbb{N}$, letting $A_0$ be a subset of $A$ with $N$ elements and denoting $T_af=f\circ T_a$, we have
\begin{align*}
N^2\left|\int_Xfd\mu\right|^2=
\frac{\left|\left\langle\sum_{a\in A_0}T_af,1\right\rangle\right|^2}{\langle1,1\rangle}&\leq\left\langle\sum_{a\in A_0}T_af,\sum_{a\in A_0}T_af\right\rangle\\
&\leq N\cdot|B|\cdot\|f\|^2_{L^2(X,\mu)}+N^2\lambda.
\end{align*}
This is a contradiction for big enough $N$, because $\lambda<|\int_Xfd\mu|^2$.
\end{proof}

The corollary below was suggested by V. Bergelson. Before stating it, recall that a subset $H$ of a countable group $G$ is \textit{thick} when for any finite $A\subseteq G$, $H$ contains a right translate of $A$. In particular, a subset of a countable amenable group has upper Banach density $1$ iff it is thick. 
Also note that a set $H$ is a vdC set in $G$ if and only if $H\cup H^{-1}$ is a vdC set in $G$; 
this follows from the fact that for any m.p.s. $(X,\mathcal{B},\mu,(T_g)_{g\in G})$ and any $f\in L^\infty(X,\mu)$ we have $\int_Xf(T_hx)\overline{f(x)}d\mu=\overline{\int_Xf(T_{h^{-1}}x)\overline{f(x)}d\mu}$.

\begin{cor}
If $G$ is a countable group and a subset $H\subseteq G\setminus\{e\}$ is thick, then $H$ is a nice vdC set.
\end{cor}

\begin{proof}
If $H\subseteq G$ is thick, then there is an infinite set $A=\{a_1,a_2,\dots\}$ such that  
$H$ contains the set $\{a_na_m^{-1};m>n\}$. Indeed, we can define $(a_n)_{n\in\mathbb{N}}$ be recursion by letting $a_n\neq a_1,\dots,a_{n-1}$ be such that $A_na_n^{-1}\subseteq H$, where $A_n:=\{a_1,\dots,a_{n-1}\}$.

Thus $H\cup H^{-1}$ contains the set $AA^{-1}$, so $H\cup H^{-1}$ is a vdC set, so $H$ is a vdC set.
\end{proof}

\section{Convexity and Cesaro averages}
\label{SecConv}

We prove in \Cref{564rtyfgdv} below that there is a close relationship between the correlation functions of sequences taking values in a compact set $D\subseteq\mathbb{C}$, and in the convex hull of $D$. We then explain several applications of this result, including an answer to a question by Kelly and Lê. The arguments involving the law of large numbers which we use to prove \Cref{564rtyfgdv} are based on \cite[Section 6]{Ru}.

\begin{prop}
\label{564rtyfgdv}
Let $G$ be a countably infinite amenable group with a F{\o}lner sequence $F=(F_N)$. Let $D\subseteq\mathbb{C}$ be compact and let $C\subseteq\mathbb{C}$ be the convex hull of $D$. Then for any sequence $(z_g)_{g\in G}$ of complex numbers in $C$ there is a sequence $(w_g)_{g\in G}$ in $D$ such that, for any $k\in\mathbb{N}$ and any pairwise distinct elements $h_1,\dots,h_k\in G$, we have
\begin{equation*}
\lim_{N\to\infty}
\frac{1}{|F_N|}
\sum_{g\in F_N}
w_{h_1g}\cdots w_{h_kg}
=
\lim_{N\to\infty}
\frac{1}{|F_N|}
\sum_{g\in F_N}
z_{h_1g}\cdots z_{h_kg},
\end{equation*}
whenever the right hand side is defined.
\end{prop}

\begin{proof}
Consider a list $(h_{l,1},\dots,h_{l,j_l})$, $l\in\mathbb{N}$, of all the finite sequences of pairwise distinct elements of $G$, and let $p_l:\mathbb{C}^{j_l}\to\mathbb{C}$; $p_l(z_1,\dots,z_{j_l})=z_1z_2\cdots z_{j_l}$.

Consider a F{\o}lner subsequence $(F_{N_i})_i$ of $F$ such that for all $l\in\mathbb{N}$, the limit 
\begin{equation*}
    \gamma(l):=
    \lim_{i\to\infty}
    \frac{1}{|F_{N_i}|}\sum_{g\in F_{N_i}}p_l(z_{h_{l,1}g},\dots,z_{h_{l,j_l}g})
    \end{equation*}
is defined. Then by \Cref{ThmFinAvsGen}, for all $A\subseteq G$ finite and for all $L\in\mathbb{N},\delta>0$ there exist some $K\in\mathbb{N}$ and some sequences $(z_{g,k})_{g\in G}$ in $C$, for $k=1,\dots,K$, such that for all $l=1,\dots,L$ we have
\begin{equation}\label{ThmFinAvsGenEq3erfds}
\left|\gamma(l)-\frac{1}{K|A|}\sum_{k=1}^{K}\sum_{g\in A}z_{h_{l,1}g,k}\dots z_{h_{l,j_l}g,k}\right|<\delta.
\end{equation}
\begin{claim}
\label{54trfgdvtrgfd}
For every sequence $\mathbf{x}=(x_g)_{g\in G}$ taking values in $C$, any finite $A\subseteq G$ and any $L\in\mathbb{N},\delta>0$ there is some $K_{\mathbf{x}}\in\mathbb{N}$ and sequences $(x_{g,k})_{g\in G}$ in $D$, $k=1,\dots,K_{\mathbf{x}}$, such that for all $g\in A$ and all $l=1,\dots,L$ we have 
\begin{equation*}
\left|\frac{1}{|A|}\sum_{g\in A}x_{h_{l,1}g}\cdots x_{h_{l,j_l}g}-\frac{1}{K_{\mathbf{x}}|A|}
\sum_{k=1}^{K_{\mathbf{x}}}x_{h_{l,1}g,k}\cdots x_{h_{l,j_l}g,k}\right|<\delta.
\end{equation*}
\end{claim}
We prove \Cref{54trfgdvtrgfd} below. It follows from \Cref{54trfgdvtrgfd} and \Cref{ThmFinAvsGenEq3erfds} that for all $A\subseteq G$ finite, $L\in\mathbb{N}$ and $\delta>0$ there exist some $K\in\mathbb{N}$ and sequences $(z_{g,k})_{g\in G}$ in $D$, for $k=1,\dots,K$, such that for all $l=1,\dots,L$ we have
\begin{equation*}
\left|\gamma(l)-\frac{1}{K|A|}\sum_{k=1}^{K}\sum_{g\in A}z_{h_{l,1}g,k}\dots z_{h_{l,j_l}g,k}\right|<\delta.
\end{equation*}
Thus, by \Cref{ThmFinAvsGen} there exists a sequence $(w_g)_{g\in G}$ of elements of $D$ such that, for all $l\in\mathbb{N}$,
    \begin{equation*}
    \lim_{N\to\infty}
    \frac{1}{|F_N|}\sum_{g\in F_N}w_{h_{l,1}g}\cdots w_{h_{l,j_l}g}
    =
    \lim_{N\to\infty}
    \frac{1}{|F_N|}\sum_{g\in F_N}p_l(w_{h_{l,1}g},\dots,w_{h_{l,j_l}g})=\gamma(l),
    \end{equation*}
    so we are done. 

    It only remains to prove \Cref{54trfgdvtrgfd}, so let $\mathbf{x}=(x_g)_{g\in G}$ take values in $C$. Note that the set of extreme points of $C$ is contained in $D$, so by Choquet's theorem, for every $x\in C$ there is a probability measure $\mu$ supported in $D$ and with average $x$. So we can consider for each $g\in G$ a random variable $\xi_g$ supported in $D$ and satisfying $\mathbb{E}(\xi_g)=x_g$. Note that, if the variables $(\xi_g)_{g\in G}$ are pairwise independent and $h_1,\dots,h_k\in G$ are distinct, then we have
    \begin{equation*}
    \mathbb{E}(\xi_{h_1}\cdots\xi_{h_k})
    =
    \mathbb{E}(\xi_{h_1})\cdots\mathbb{E}(\xi_{h_k})=x_{h_1}\cdots x_{h_k}.
    \end{equation*}
    
    Now for each $k\in\mathbb{N}$ we will choose a sequence $(x_{g,k})_{g\in G}$, where the variables $x_{g,k}$ are chosen pairwise independently and with distribution $\xi_g$. Then, for all $A\subseteq G$ finite, $L\in\mathbb{N}$ and $\delta>0$, by the strong law of large numbers we will have with probability $1$ that, for big enough $K$,
    \begin{equation*}
\left|\frac{1}{|A|}\sum_{g\in A}x_{h_{l,1}g}\cdots x_{h_{l,j_l}g}-\frac{1}{K|A|}
\sum_{k=1}^Kx_{h_{l,1}g,k}\cdots x_{h_{l,j_l}g,k}\right|<\delta,
\end{equation*}  
    concluding the proof.\qedhere
\end{proof}

Applying \Cref{564rtyfgdv} to the set $D=\{0,1\}$ gives the following result.

\begin{prop}
\label{FunctionsToSetsGroup}
Let $G$ be a countably infinite amenable group with a F{\o}lner sequence $F=(F_N)$. Then for any sequence $(z_g)_{g\in G}$ of numbers in $[0,1]$ there is a set $B\subseteq G$ such that, for any $k\in\mathbb{N}$ and any pairwise distinct elements $h_1,\dots,h_k\in G$, we have
\begin{equation*}
d_F({h_1}^{-1}B\cap\dots\cap {h_k}^{-1}B)
=
\lim_{N\to\infty}
\frac{1}{|F_N|}
\sum_{g\in F_N}
z_{h_1g}\cdots z_{h_kg},
\end{equation*}
whenever the right hand side is defined.\qed
\end{prop}

It follows that, given a measurable function $f:X\to[0,1]$, we can obtain a function $g:X\to\{0,1\}$ (that is, a measurable set) with the same correlation functions: 

\begin{prop}[Turning functions into sets]
\label{FunctionsToSets}
Let $G$ be a countably infinite amenable group. For every m.p.s. $(X,\mathcal{B},\mu,(T_g)_{g\in G})$ and every measurable $f:X\to[0,1]$ there exists a m.p.s. $(Y,\mathcal{C},\nu,(S_g)_{g\in G})$ and $B\in\mathcal{B}$ such that for all $k\in\mathbb{N}$ and all distinct $h_1,\dots,h_k\in G$ we have
\begin{equation*}
\nu\left(T_{h_1^{-1}}B\cap\dots\cap T_{h_k^{-1}}B\right)
=
\int_Xf(T_{h_1}x)\cdots f(T_{h_k}x)d\mu.
\end{equation*}
\end{prop}

\begin{proof}
Let $(X,\mathcal{B},\mu,(T_g)_{g\in G}), f$ be as in \Cref{FunctionsToSets}, and fix a F{\o}lner sequence $F=(F_N)$ in $G$. Thanks to \Cref{IFC} there is a $[0,1]$-valued sequence $(z_g)_{g\in G}$ such that, for all $k\in\mathbb{N}$ and $h_1,\dots,h_k\in G$,
\begin{equation*}
\lim_N\frac{1}{|F_N|}\sum_{g\in F_N}z_{h_1g},\dots,z_{h_jg}=\int_Xf(T_{h_1}x),\dots,f(T_{h_j}x)d\mu.
\end{equation*}
Thus, by \Cref{FunctionsToSetsGroup} there is a set $B\subseteq G$ such that, for all $k\in\mathbb{N}$ and $h_1,\dots,h_k\in G$,
\begin{equation*}
d_F({h_1}^{-1}B\cap\dots\cap {h_k}^{-1}B)=\int_Xf(T_{h_1}x),\dots,f(T_{h_j}x)d\mu.
\end{equation*}
Applying $\Cref{FCP}$ to the set $B$, we are done.
\end{proof}

\begin{remark}
The elements $h_1,\dots,h_k$ need to be distinct in \Cref{FunctionsToSets}; if we applied the result to $h_1=h_2=e$, we would obtain $\nu(B)=\int_Xfd\mu=\int_Xf^2d\mu$, which only happens if $f$ is essentially a characteristic function.
\end{remark}

Note that the statement of \ref{FunctionsToSets} makes sense for non-amenable groups. It seems plausible to us that \Cref{FunctionsToSets} could be proved for any group $G$ using some purely measure-theoretic construction:

\begin{question}
Is \Cref{FunctionsToSets} true for any countable group $G$?
\end{question}

The last result in this section, which we prove after \Cref{564rtyfgdv}, is about sequences of norm $1$ which are `very well distributed' in $\mathbb{S}^1$:

\begin{prop}[White noise]
\label{WhiteNoise}
Let $G$ be a countable amenable group with a F{\o}lner sequence $(F_N)$. There is a sequence $(z_g)_{g\in G}$ of complex numbers in $\mathbb{S}^1$ such that, for all $k\in\mathbb{N},n_1,\dots,n_k\in\mathbb{Z}\setminus\{0\}$ and distinct elements $h_1,\dots,h_k\in G$ we have
\begin{equation}
\label{EqWhiteNoise}
\lim_{N\to\infty}
\frac{1}{|F_N|}
\sum_{g\in F_N}
z_{h_1g}^{n_1}\cdots z_{h_kg}^{n_k}
=0.
\end{equation}
\end{prop}

\begin{proof}
By \Cref{ThmFinAvsGen}, \ref{ThmFinAvsGen3}$\implies$\ref{ThmFinAvsGen1}, it is enough to prove that for all $A\subseteq G$ finite, $h_1,\dots,h_j\in G$  and for all $L\in\mathbb{N},\delta>0$ there is some $K\in\mathbb{N}$ and sequences $(z_{g,k})_{g\in G}$ in $D$, for $k=1,\dots,K$, such that we have
    \begin{equation}
    \label{34rwefsd9o90okk}
    \left|\frac{1}{K|A|}\sum_{k=1}^{K}\sum_{g\in A}z_{h_{1}g,k}^{l_1},\dots,z_{h_{j}g,k}^{l_j}\right|<\delta\textup{  for all }l_1,\dots,l_j\in\mathbb{Z}\setminus\{0\},|l_i|\leq L.
    \end{equation}

Consider an independent sequence of random variables $(\xi_g)_{g\in G}$, such that $\xi_g$ is uniformly distributed in $\mathbb{S}^1$ for all $g\in G$. Clearly, for all $k\in\mathbb{N},n_1,\dots,n_k\in\mathbb{Z}\setminus\{0\}$ and distinct $h_1,\dots,h_k\in G$ we have 
\begin{equation*}
\mathbb{E}(\xi_{h_1}^{n_1}\cdots\xi_{h_k}^{n_k})=0.
\end{equation*}

So, as in the proof of \Cref{564rtyfgdv}, for each $k\in\mathbb{N}$ we choose a sequence $(x_{g,k})_{g\in G}$, where the variables $x_{g,k}$ are chosen pairwise independently and with distribution $\xi_g$. The strong law of large numbers then implies that for big enough $K$, \Cref{34rwefsd9o90okk} will be satisfied with high probability, concluding the proof.
\end{proof}

\begin{remark}
If the sequence $(F_N)$ from \Cref{WhiteNoise} does not grow very slowly (e.g. if we have $\sum_N\alpha^{|F_N|}<\infty$ for all $\alpha<1$), then we can choose each of the elements $z_g$ from the sequence $(z_g)_{g\in G}$ independently and uniformly from $\mathbb{S}^1$. The sequence we obtain will satisfy \Cref{WhiteNoise} with probability $1$.
\end{remark}

\begin{remark}
Let $G=\mathbb{Z}$. If we consider only finitely many values of $k$, then we can construct sequences $(z_n)_{n\in\mathbb{Z}}$ which satisfy \Cref{EqWhiteNoise} for all F{\o}lner sequences simultaneously. For example, letting $z_n=e^{n^8\alpha}$, where $\alpha\in\mathbb{R}\setminus\mathbb{Q}$, one can check that \Cref{EqWhiteNoise} is satisfied for all $k\leq8$, for all $n_i,h_i$ as above and, most importantly, for all F{\o}lner sequences $(F_N)$. 
However, there is no sequence $(z_n)_{n\in\mathbb{Z}}$ which satisfies \Cref{EqWhiteNoise} for all $k\in\mathbb{N}$ and for all F{\o}lner sequences $F$ and $(n_i)_{i=1}^k,(h_i)_{i=1}^k$ as above. 
Indeed, suppose that such a sequence $(z_n)_{n\in\mathbb{Z}}$ exists. Then for all $N\in\mathbb{N}$, the sequence $n\mapsto(z_{n+1},\dots,z_{n+N})$ has to be u.d. in $(\mathbb{S}^1)^N$ (see \cite[Chapter 1, Theorem 6.2]{KN}). 
But that implies that for some $n_N\in\mathbb{N}$, the numbers $z_{n_N+1},\dots,z_{n_N+N}$ all have positive real part. Thus, letting $F_N=\{n_N+1,\dots,n_N+N\}$, the sequence $(z_n)$ is not $(F_N)$-u.d., a contradiction.
\end{remark}

As was pointed out to us by S. Farhangi, we can use \Cref{564rtyfgdv} to answer a question of Kelly and Lê. We first need some definitions.

\begin{definition}
We say a sequence $(x_n)_{n\in\mathbb{N}}$ of elements of a compact topological group $G$, with Haar measure $\mu_G$, is u.d. in $G$ if
\begin{equation*}
\lim_N\frac{1}{N}|\{n\in\{1,\dots,N\};x_n\in C\}|=\mu_G(C)
\end{equation*}
for all open sets $C\subseteq G$ with boundary of measure $0$.
\end{definition}

\begin{definition}
Let $G$ be a compact topological group. We say a set $H\subseteq\mathbb{N}$ is $G$-u.d.vdC\footnote{We use the notation `$G$-u.d.vdC' to avoid confusion with `vdC in $G$', as in \Cref{DefGvdC}.} if for any sequence $(x_n)_{n=1}^\infty$ in $G$ such that $(x_{n+h}x_n^{-1})_{n\in\mathbb{N}}$ is u.d. in $G$ for all $h\in H$, the sequence $(x_n)_{n\in\mathbb{N}}$ is also u.d. in $G$.
\end{definition}

In \cite[Page 2]{KL} the authors mention that it is an interesting (and perhaps difficult) problem to determine whether all $\mathbb{Z}_2$-u.d.vdC sets are vdC. 

\begin{prop}
\label{4rewfdso90oplk}
A set $H\subseteq\mathbb{N}$ is vdC iff it is $\mathbb{Z}_2$\textup{-u.d.vdC.}
\end{prop}

\Cref{4rewfdso90oplk} follows from \Cref{rewf9ds0ol} and the observation that a sequence $(x_n)_{n\in\mathbb{N}}$ is u.d. in $\mathbb{Z}_2$
 iff $\lim_{N}\frac{1}{N}\sum_{n=1}^N(-1)^{x_n}=0$.

\begin{prop}
\label{rewf9ds0ol01}
Let $H\subseteq\mathbb{N}$. $H$ is vdC iff for every sequence $(x_n)_{n\in\mathbb{Z}}$ with $x_n\in[-1,1]$ for all $n$ we have
\begin{equation}
\label{4r9efdsoixckloi}
\lim_{N\to\infty}\frac{1}{N}
    \sum_{n=1}^Nx_{n+h}x_n=0\text{ for all }h\in H\textup{ implies }
    \lim_{N\to\infty}\frac{1}{N}\sum_{n=1}^Nx_n=0.
\end{equation}
\end{prop}

\begin{proof}
The following argument was provided by S. Farhangi. $\implies$ follows from \Cref{ThmvdC2} in \Cref{ThmvdC}. So let us prove that, if every sequence $(x_n)_{n\in\mathbb{Z}}$ in $[-1,1]$ satisfies \Cref{4r9efdsoixckloi}, then every sequence $(x_n)_{n\in\mathbb{Z}}$ in $\mathbb{D}$ satisfies \Cref{4r9efdsoixckloi}, with $x_{n+h}\overline{x_n}$ instead of $x_{n+h}x_n$ (so $H$ is vdC by \Cref{ThmvdC}). Thanks to \Cref{ThmFinAvs2.0}, it will be enough to prove the measure theoretic analog of our result: 
we will assume that for all m.p.s. $(X,\mathcal{B},\mu,(T_g)_{g\in G})$ and measurable $f:X\to[-1,1]$ we have
\begin{equation}
\label{refsd9opl99oplk}
\int_{X}f(T_hx)\cdot\overline{f(x)}d\mu(x)=0\text{ for all }h\in H\text{ implies }\int_{X}fd\mu=0,
\end{equation} 
and we will prove that \Cref{refsd9opl99oplk} holds true for all m.p.s. $(X,\mathcal{B},\mu,(T_g)_{g\in G})$ and measurable $f:X\to\mathbb{D}$.

So, fix $(X,\mathcal{B},\mu,(T_g)_{g\in G})$ and $f:X\to\mathbb{D}$ such that $\int_{X}f(T_hx)\cdot\overline{f(x)}d\mu(x)=0$ for all $h\in H$. We let $f_0,f_1:X\to[-1,1]$ be the real and imaginary parts of $f$ respectively.
Consider the set $Y=X\times\{0,1\}$ with the product $\sigma$-algebra $\mathcal{C}=\mathcal{B}\times\mathcal{P}(\{0,1\})$, measure $\nu(A_0\times\{0\}\cup A_1\times\{1\})=\frac{\mu(A_0)+\mu(A_1)}{2}$ for $A_0,A_1\in\mathcal{B}$, and action $(S_g)_{g\in G}$ given by $S_g(x,i)=(T_gx,i)$ for $x\in X,i\in\{0,1\}$.

Finally, let $F:Y\to[-1,1];F(x,i)=f_i(x)$. Then for all $h\in H$,
\begin{align*}
\int_YF(S_hy)F(y)d\nu(y)&=\frac{1}{2}\int_Xf_0(T_hx)f_0(x)+f_1(T_hx)f_1(x)d\mu(x)\\
&=\frac{1}{2}\textup{Re}\left(\int_Xf(T_hx)\overline{f(x)}d\mu(x)\right)=0.
\end{align*}
\Cref{refsd9opl99oplk} then implies $\int_YFd\mu=0$, so
$
\int_Xf_0+f_1
d\mu=0$. Similarly, we can conclude $
\int_X-f_1+f_0d\mu=0$ by applying the same reasoning to the function $g=if=-f_1+if_0$ instead of $f$.
Thus, $\int_Xf_0d\mu=\int_Xf_1d\mu=0$, so $\int_Xfd\mu=0$.
\end{proof}

\begin{prop}
\label{rewf9ds0ol}
A set $H\subseteq\mathbb{N}$ is vdC iff for every sequence $(x_n)_{n\in\mathbb{Z}}$ with $x_n\in\{-1,1\}$ for all $n$ we have
\begin{equation*}
\lim_{N\to\infty}\frac{1}{N}
    \sum_{n=1}^Nx_{n+h}x_n=0\text{ for all }h\in H\textup{ implies }
    \lim_{N\to\infty}\frac{1}{N}\sum_{n=1}^Nx_n=0.
\end{equation*}
\end{prop}

\begin{proof}
$\implies$ follows from \Cref{ThmvdC2} in \Cref{ThmvdC}. So suppose $H$ is not vdC. Then by \Cref{rewf9ds0ol01} there is a sequence $(x_n)_{n\in\mathbb{Z}}$ in $[-1,1]$, some $\lambda\neq0$ and some F{\o}lner subsequence $(F_N)$ of $(\{1,\dots,N\})$, such that 
\begin{equation*}
\lim_{N\to\infty}\frac{1}{|F_N|}
    \sum_{n\in F_N}x_{n+h}x_n=0\text{ for all }h\in H\textup{ but }
    \lim_{N\to\infty}\frac{1}{|F_N|}\sum_{g\in F_N}x_{g}=\lambda.
\end{equation*}
Thus, by \Cref{564rtyfgdv} there is a sequence $(w_n)_{n\in\mathbb{Z}}$ in $\{-1,1\}$ such that 
\begin{equation*}
\lim_{N\to\infty}\frac{1}{|F_N|}
    \sum_{n\in F_N}w_{n+h}w_n=0\text{ for all }h\in H\textup{ but }
    \lim_{N\to\infty}\frac{1}{|F_N|}\sum_{n\in F_N}w_{n}=\lambda.
\end{equation*}
So by \Cref{ThmFinAvsGen} (which implies that the set of correlation functions of sequences in a given compact set is independent of the F{\o}lner sequence), there is a sequence $(u_n)_{n\in\mathbb{Z}}$ in $\{-1,1\}$ such that 
\begin{equation*}
\lim_{N\to\infty}\frac{1}{N}
    \sum_{n=1}^Nu_{n+h}u_n=0\text{ for all }h\in H\textup{ but }
    \lim_{N\to\infty}\frac{1}{N}\sum_{n=1}^Nu_{n}=\lambda.\qedhere
\end{equation*}
\end{proof}


\bibliographystyle{alpha}
\bibliography{Bibliography}
\end{document}